\documentclass[11pt]{article}
\usepackage{graphicx}
\usepackage{cite}
\usepackage{color}
\usepackage{amsmath,amssymb,amsthm,amsfonts}
\usepackage{amssymb}
\newtheorem{thm}{Theorem}[section]

\newtheorem{lem}[thm]{Lemma}
\newtheorem{prop}[thm]{Proposition}

\theoremstyle{definition}

\theoremstyle{remark}

\numberwithin{equation}{section}

\DeclareMathSymbol{\C}{\mathalpha}{AMSb}{"43}
\usepackage{amsmath}
\allowdisplaybreaks[4]

\newcommand{\DETAILS}[1]{}

\newcommand{\eps}{\varepsilon}

\newcommand{\lam}{\lambda}

\newcommand{\R}{{\mathbb{R}}}

\newcommand{\inte}{\int_{\mathbb{R}^3}}

\DeclareMathOperator{\tr}{Tr}

\newcommand{\bsub}{\begin{subequations}}
\newcommand{\esub}{\end{subequations}$\!$}

\begin{document}
\title{On the Limiting Behavior of $L^2$-Critical Pseudo-Relativistic Fermi Systems}

\author{
Bin Chen\thanks{School of Mathematics and Statistics,  Key Laboratory of Nonlinear Analysis $\&$ Applications (Ministry of Education), Central China Normal University, Wuhan 430079, P. R. China; and State Key Laborotary of Mathematical Sciences, Academy of Mathematics and Systems Science, Chinese Academy of Sciences, Beijing 100190, P. R. China. B. Chen  is partially supported by NSF of China (Grant 12501151). Email: binchen@amss.ac.cn.}, Yinbin Deng\thanks{School of Mathematics and Statistics, Hubei Key Laboratory of Mathematical Sciences, Central China Normal University, Wuhan 430079, P. R. China. Y. B. Deng is partially supported by NSF of China (Grant 12271196) and National Key R $\&$ D Program of China (Grant 2023YFA1010002). Email: ybdeng@ccnu.edu.cn.},
Yujin Guo\thanks{School of Mathematics and Statistics,  Key Laboratory of Nonlinear Analysis $\&$ Applications (Ministry of Education), Central China Normal University, Wuhan 430079, P. R. China. Y. J. Guo is partially supported by NSF of China (Grants 12225106 and 12371113) and National Key R $\&$ D Program of China (Grant 2023YFA1010001). Email: yguo@ccnu.edu.cn.},
\ and Chenyang Wang\thanks{School of Mathematics and Statistics,  Key Laboratory of Nonlinear Analysis $\&$ Applications (Ministry of Education), Central China Normal University, Wuhan 430079, P. R. China. Email:  wangcy789@mails.ccnu.edu.cn.}
}

\date{\today}

\smallbreak \maketitle

\begin{abstract}
We consider   ground states of a pseudo-relativistic Fermi system in the $L^2$-critical case. We prove that the system admits ground states, if and only if the attractive strength $a$ satisfies $0<a<D_{4/3,2}$, where  $D_{4/3,2}\in(0, \infty)$ is the optimal constant of a dual fractional Lieb--Thirring inequality. The limiting behavior of ground states for the system is further analyzed as $a\nearrow D_{4/3,2}$. As a byproduct, the qualitative properties of  optimizers for the dual fractional Lieb-Thirring inequality are also investigated.
\end{abstract}
\vskip 0.05truein


\vskip 0.2truein
\section{Introduction}
The quantum many-body problems have received  a lot of attentions since they were proposed as  rigorous mathematical models in 1926 by E.  Schr\"odinger (cf. \cite{ES26}). 
It is well known that  all elementary particles in quantum mechanics are fundamentally classified into two categories  in terms of their spin quantum numbers: bosons and fermions.
Identical bosons can occupy the exact same quantum state, whereas  fermions obey the Pauli exclusion principle, which prohibits two identical fermions from occupying  simultaneously the same quantum state. This leads to the fact that studying ground states of Fermi systems is  more challenging than studying ground states of  bosonic systems.
Ground states of bosonic systems have been studied extensively since the past few decades, see \cite{Lion84b, MM,Lieb87,Eh09} and the references therein. In spite of this fact,  as far as we know, there exist however fewer analysis for ground states of fermionic systems, especially in the pseudo-relativistic case, see \cite{Duke,Lieb87,JPS10, cb3} and the references therein.  We also  refer to \cite{cb1, Finite, Arma, CMP, Geo} and the references therein for the analysis of nonrelativistic fermionic systems.

In this paper, we consider ground states of the following $L^2$-critical  pseudo-relativistic Fermi system:
\begin{align}\label{pro11}
E_{a}(N):=\inf \Big\{ \mathcal{E}_{a}(\Psi): \ \|\Psi\|_{2}^{2}=1, \Psi\in \land^{N} L^{2}(\R^{3}, \C)\cap  H^{\frac{1}{2}}(\R^{3N},\C)
\Big\},\ \ a>0,
\end{align}
where $N\in\mathbb{N}^+$ denotes the number of  spinless fermions,  the energy functional $\mathcal{E}_{a}(\Psi)$  is defined by
\begin{equation*}\label{a16}
\mathcal{E}_{a}(\Psi):=\sum_{i=1}^{N}\Big\langle \Psi, \,   \big(\sqrt{-\Delta_{x_{i}}+m^{2}}-m\big)\Psi\Big\rangle -a\int_{\R^{3}}\rho_{\Psi}^{\frac{4}{3}}(x)dx,\ \ m>0,
\end{equation*}
and  $\Psi\in \land^{N} L^{2}(\R^{3}, \C)$ is an antisymmetric wave function. Here the pseudo-differential operator $\sqrt{-\Delta+m^2}-m$ describes the kinetic energy of relativistic fermions with rest mass $m>0$,
the parameter $a>0$ represents the attractive strength of the interactions among the fermions,   and the one-particle density $\rho_{\Psi}$ of $\Psi$ is defined as
\begin{align*}
\rho_{\Psi}(x):=N\int_{\R^{3(N-1)}}|\Psi(x, x_{2}, \cdots, x_{N})|^{2}dx_{2}\cdots dx_{N}.
\end{align*}
By the  density functional argument (cf. \cite{cb1,Arma}), the problem (\ref{pro11}) can be reduced equivalently to the following form
\begin{align}\label{problem}
E_{a}(N)=\inf \Big\{ \mathcal{E}_{a}(\gamma):&\
\gamma=\sum_{i=1}^{N}|u_{i}\rangle \langle u_{i}|,\,\ u_{i}\in H^{\frac{1}{2}}(\R^{3},\C),\\
&\ \langle u_{i},u_{j}\rangle_{L^{2}}=\delta_{ij}, \ i,j=1, \cdots, N \Big\},\ \ a>0,\ N\in\mathbb{N}^+,\nonumber
\end{align}
where the functional $\mathcal{E}_{a}(\gamma)$ is defined as
\begin{align}\label{EVdf}
\mathcal{E}_{a}(\gamma):=\tr(\sqrt{-\Delta+m^{2}}-m)\gamma-a\int_{\R^{3}}\rho_{\gamma}^{\frac{4}{3}}dx,\ \ \ m>0,
\end{align}
$\gamma=\sum_{i=1}^{N}|u_{i}\rangle \langle u_{i}|$ denotes an $N$-dimensional orthogonal projection operator on  $L^{2}(\R^{3},\C)$, namely,
\begin{align*}
\big(\gamma\varphi\big)(x)=\sum_{i=1}^{N}u_{i}(x)\inte \varphi(y) \bar{u}_i(y)dy,\ \ \  \forall\  \varphi\in  L^{2}(\R^{3},\C),
\end{align*}
and  $\rho_{\gamma}(x):=\sum_{i=1}^{N}|u_{i}(x)|^{2}$  denotes the corresponding density of $\gamma$.

Following the equivalence \eqref{problem}, throughout this paper we thus focus on the analysis of the constrained minimization problem \eqref{problem}, instead  of \eqref{pro11},  where the pseudo-differential operator $\sqrt{-\Delta+m^{2}}$ is defined via the Fourier transform, $i.e.$, for any $u\in H^{\frac{1}{2}}(\R^{3},\C)$,
\begin{align*}	(\sqrt{-\Delta+m^{2}}u)\string^(\xi)=\sqrt{|\xi|^{2}+m^{2}}\hat{u}(\xi), \  \ \  \hat{u}(\xi)=\int_{\R^{3}}u(x)e^{-2\pi i \xi\cdot x}dx.
\end{align*}
Note that the  minimization problem $E_{a}(1)$ is essentially an $L^{2}$-critical boson problem.
Further, it was proved in \cite{2017siam,LMP2014} that when an external potential $V(x)$ is imposed to the non-relativistic $L^2$-critical bosonic problem $E_{a}(1)$,  the corresponding system admits ground states if and only if $0<a<a^*$, where  $0<a^*<\infty$ is a critical strength of the attractive interactions for the system.  Moreover,  the critical   strength $a^{*}>0$ is determined by the unique optimizer of the classical Gagliardo--Nirenberg (GN) inequality  studied in \cite{CMP1983}.

The purpose of this paper is to analyze the existence and limiting behavior of minimizers for the fermionic problem  $E_{a}(N)$ with $N=2$. It turns out to be  closely connected with the following  Gagliardo-Nirenberg inequality of orthonormal systems: for $N\in\mathbb{N}^+$,
$$
\|\gamma\|^{1/3}\tr(\sqrt{-\Delta}\gamma)\geq D_{4/3,N}\inte \rho_\gamma^{4/3} dx,\ \,  \forall\ \gamma\in \mathcal{R}_{N},
$$
where
\begin{align}\label{RNdef}
\mathcal{R}_{N}:=\Big\{\gamma\in \mathcal{B}(L^{2}(\R^{3}, \C))\setminus \{0\}: \, 0\le \gamma=\gamma^{*},\  \text{Rank}(\gamma)\le N,\ \tr(\sqrt{-\Delta}\gamma)<\infty      \Big\},
\end{align}
and $\mathcal{B}(L^{2}(\R^{3}, \C))$ denotes the set of bounded linear operators on $L^{2}(\R^{3}, \C)$.  Applying the spectral theorem (cf. \cite{Arma}), for any $\gamma\in\mathcal{R}_{N}$, there exist an orthonormal system $\{u_i\}_{i=1}^{R_N}\subset L^{2}(\R^{3}, \C)$ and a sequence $\{n_i\}_{i=1}^{R_N}\subset \R^+$ such that
$$
\gamma=\sum_{i=1}
^{R_N}n_i|u_i\rangle\langle u_i|,\  \,   R_N\in[1, N],
$$
in the sense that
\begin{align*}
\big(\gamma\varphi\big)(x)=\sum_{i=1}^{R_{N}}n_iu_{i}(x)\inte \varphi(y) \bar{u}_i(y)dy,\ \,    \forall\  \varphi\in  L^{2}(\R^{3},\C).
\end{align*}
More generally, we consider the following dual fractional Lieb--Thirring inequality: for $1< p\le \frac{4}{3}$ and $N\in\mathbb{N}^+$,
\begin{align}\label{LTineq}
\|\gamma\|_{\mathfrak{S}^{q}}^{\frac{3-2p}{3(p-1)}}\tr(\sqrt{-\Delta}\gamma)\geq D_{p,N}\|\rho_{\gamma}\|_{L^{p}(\R^{3})}^{\frac{p}{3(p-1)}},\ \, \forall \ \gamma\in \mathcal{R}_{N},
\end{align}
where $D_{p,N}\geq0$ denotes the best constant, and
\begin{align}\label{1.5}
\|\gamma\|_{\mathfrak{S}^{q}}:=
\left\{
\begin{array}{lll}
\!\!\big(\tr \gamma^{q}\big)^{\frac{1}{q}},\  &\text{if}\ \ 1\le q<\infty; \\[1mm]
\!\!	\|\gamma\|,\ \ &\text{if}\ \ q=\infty
\end{array}\right.
\end{align}
satisfying
\begin{align}\label{q}
q:=\left\{\begin{array}{lll}
\!\!\dfrac{3-2p}{4-3p},\  &\text{if}\ \ 1< p<\frac{4}{3}; \\[2mm]
\!\!+\infty,\ \ &\text{if}\ \ p=\frac{4}{3}
\end{array}\right.
\end{align}
denotes the $q$-th Schatten  norm of $\gamma$.

\subsection{Main results}

The purpose of this subsection is to introduce the main results of the present paper.
Our first result is concerned with the following existence and nonexistence of minimizers for the problem $E_{a}(2)$ defined in (\ref{problem}).

\begin{thm}\label{th21}
Let $E_{a}(2)$  and  $D_{4/3,2}$ be defined by $(\ref{problem})$  and $(\ref{LTineq})$, respectively. Then we have the following conclusions:
\begin{enumerate}
\item [(1)]  If $0< a < D_{4/3,2}$, then $E_{a}(2)$ admits at least one minimizer $\gamma=\sum_{i=1}^{2}|u_{i}\rangle \langle u_{i}|$, where the orthonormal system  $(u_{1} ,u_{2})$ satisfies  the following fermionic nonlinear Schr\"odinger system
\begin{align}\label{1.18}
H_{\gamma}u_{i}:=\Big[\sqrt{-\Delta+m^{2}}-m-\frac{4a}{3}\rho_\gamma^{\frac{1}{3}}\Big]u_{i}=\mu_{i}u_{i}\ \ \text{in}\ \, \R^{3}, \ \ i=1, 2.
\end{align}
Here $\mu_{1}<\mu_{2}<0$ are the first two  eigenvalues  of the operator $H_{\gamma}$ on  $L^2(\R^{3},\C)$.

\item [(2)]  If $a\ge D_{4/3,2}$, then $E_{a}(2)$ does not admit any minimizer.
\end{enumerate}
\end{thm}

Theorem \ref{th21} gives a complete classification on the existence and nonexistence of minimizers for $E_{a}(2)$ in terms of $a>0$.  We remark that the proof of Theorem \ref{th21} can be extended naturally to  the more general problem $E_{a}(N)$, where $ N\in\mathbb{N}^{+}$ satisfies $D_{4/3,N-1}>D_{4/3,N}$.
Compared with the existing results in \cite{cb3,Arma}, where the existence of minimizers for $L^2$-subcritical Fermi systems was studied, it is more challenging to analyze the non-existence of Theorem \ref{th21}. More precisely, the non-existence of Theorem \ref{th21} relies on   the energy estimates of $E_{a}(N)$ given in Lemma \ref{lem32},  whereas the existence of Theorem \ref{th21} is established by proving the strict binding inequality $E_{a}(2)<2E_{a}(1)$. Towards the proof of Theorem \ref{th21},  we need to make full use of the following qualitative properties of the optimizers for $D_{p,N}$ defined by \eqref{LTineq}.

\begin{prop}\label{th11}  Let $D_{p,N}$ be defined by $(\ref{LTineq})$, where $1< p\le \frac{4}{3}$ and $N\in \mathbb{N}^{+}$. Then we have
\begin{enumerate}
\item [(1)] The best constant  $D_{p,N}\in(0, \infty)$ can be attained.

\item  [(2)] Any optimizer $\gamma$ of $D_{p,N}$ satisfying $\|\gamma\|_{\mathfrak{S}^{q}}=1$ has rank $R_{N}\!\in\![1,  N]$ and can be written as
\begin{equation}\label{gaform}
\gamma=\sum_{j=1}^{R_{N}}n_{j}|w_{j}\rangle \langle w_{j}|,\ \,  \langle w_{j}, w_{k}\rangle=\delta_{jk}\ \  \text{for}\ \  i,j=1,\cdots, R_{N},
\end{equation}
where $q\in(0, +\infty]$ is given by  \eqref{q},
\begin{align}\label{njeq}
n_{j}\left\{
\begin{array}{lll}
\!\!\!\in(0, 1],\ \ &\text{if}\ \ 1< p<\frac{4}{3}; \\[3mm]
\!\!\! =1,\ \ &\text{if}\ \ p=\frac{4}{3},
\end{array}\right.
\end{align}
and  the  orthonormal system $(w_{1},\cdots ,w_{R_{N}})$ solves the following  nonlinear system
\begin{equation}\label{Qeq}
	H^*_{\gamma}w_{j}:=\Big(\sqrt{-\Delta}-pC(p, N)\rho_{\gamma}^{p-1}\Big)w_{j}=\mu_{j}w_{j}\ \, \ in \ \R^3,\  \, j=1,\cdots,R_{N}.
\end{equation}
Here $C(p, N)=\frac{D_{p,N}^{3(p-1)}}{3(p-1) \big(\tr(\sqrt{-\Delta}\gamma)\big)^{3p-4}}>0$, and $\mu_1<\mu_2\le \cdots\leq\mu_{R_N}<0$ are  the $R_{N}$ first eigenvalues, counted with multiplicity, of the operator $H^*_{\gamma}$ on $L^2(\R^3,\C)$.

\item [(3)]  There exists an infinite sequence of integers $N_{1}=1<N_{2}=2<N_{3}<\cdots$ such that the rank of any optimizer of  $D_{p,N_{k}}$  is $N_{k}$, and
\begin{align}\label{1.10}
	D_{p,N_{k}}<D_{p,N_{k}-1},\ \, k\geq2.
\end{align}
\end{enumerate}
\end{prop}

The proof of Proposition \ref{th11} is stimulated by \cite{CMP} and the references therein, which handle mainly the non-relativistic fermionic systems. In spite of this fact, compared with the aforementioned works, since the inequality   $(\ref{LTineq})$ involves the relativistic operator $\sqrt{-\Delta}$, one needs to employ a different argument for improving the regularity of optimizers  and proving the strict monotonicity \eqref{1.10}. Moreover,  because we only require the optimizer $\gamma$ to satisfy the general normalization condition $\|\gamma\|_{\mathfrak{S}^q} = 1$, more involved calculations are needed for analyzing the explicit expression of $\gamma$.

By applying Proposition \ref{th11}, we next investigate the following limiting behavior of minimizers for $E_{a}(2)$ as $a\nearrow D_{4/3,2}$. 

\begin{thm}\label{th22}
Let $\gamma_{a}=\sum_{i=1}^{2}|u^{a}_{i}\rangle \langle u^{a}_{i}|$ be a minimizer of $E_{a}(2)$ satisfying \eqref{1.18} as  $a\nearrow D_{4/3,2}$.   Then there exist a subsequence $\{\gamma_{a_n}\}$ of $\{\gamma_a\}$ and a sequence $\{y_{a_n}\}\subset \R^{3}$, where $a_n\nearrow D_{4/3,2}$ as $n\to\infty$, such that for $i=1,2$,
\begin{align}\label{cv21}
\begin{split}
w_{i}^{a_{n}}(x):=\eps_{a_{n}}^{\frac{3}{2}}u_{i}^{a_{n}}\big(\eps_{a_{n}}(x+y_{a_n})\big)
\to w_{i}(x)\quad \\
 \text{strongly in} \, \ H^{\frac{1}{2}}(\R^{3}, \C)\cap L^{\infty}(\R^{3}, \C)\ \, n\to \infty,
\end{split}
\end{align}
where  $\gamma:=\sum_{i=1}^{2}|w_{i}\rangle \langle w_{i}|$ is an optimizer of the variational problem
\begin{equation*}
d_{*}:=\inf \big\{ \tr\frac{\gamma}{\sqrt{-\Delta}}: \ \gamma \   \text{is an optimizer of}\ D_{4/3,2} \ \text{satisfying}\  \|\gamma\|=\tr\big(\sqrt{-\Delta}\gamma\big)=1
\big\},
\end{equation*}
and $\eps_{a_{n}}:=(\tr\sqrt{-\Delta}\gamma_{a_{n}})^{-1}>0$ satisfies
\begin{align}\label{epsd}
\eps_{a_{n}}\sim \Big[\frac{2(D_{4/3,2}-a_{n})}{D_{4/3,2}m^{2}d_{*}}\Big]^{\frac{1}{2}}\ \ \text{as}\ \ n\to \infty.
\end{align}
\end{thm}

The proof of Theorem \ref{th22} needs the above auxiliary minimization problem $d_{*}$, for which we shall prove in Lemma \ref{lem41} that $d_{*}\in (0,\infty)$ can be  attained. Even though the function $w_{i}^{a_{n}}$ decays exponentially as $n\to\infty$, we shall  prove in Lemma  \ref{lem44} that the sequence  $\{w_{i}^{a_{n}}\}_n$ just satisfies the uniform algebraic decay as $n\to\infty$ by establishing the $H^{\frac{1}{2}}$-convergence of $w_{i}^{a_{n}}$ as $n\to\infty$, where  $i=1,2$.
Applying above results, together with the refined analysis of the energy $E_{a_{n}}(2)$ as $n\to\infty$, in Section 4 we are then able to establish \eqref{epsd} and as well the following energy estimate:
\begin{equation}\label{spi}
E_{a}(2)+2m\sim \sqrt{2}m\Big(\frac{D_{4/3,2}-a}{D_{4/3,2}}d_{*}\Big)^{\frac{1}{2}}\ \ \ \text{as}\ \ a\nearrow D_{4/3,2}.
\end{equation}


This paper is organized as follows. In Section 2, we establish Proposition \ref{th11} on the qualitative properties of optimizers for the dual fractional Lieb-Thirring inequality (\ref{LTineq}).  In Section 3 we then prove Theorem \ref{th21} on the existence and nonexistence of minimizers for the problem $E_a(2)$.  Section 4 is devoted to the proof of Theorem \ref{th22}, which is concerned with the limiting behavior of minimizers for $E_{a}(2)$ as $a\nearrow D_{4/3,2}$.

\section{A Dual Fractional Lieb--Thirring Inequality}\label{lt}
Let   $D_{p,N}\geq0$  be the best constant of the dual fractional Lieb--Thirring inequality \eqref{LTineq}, where $1<p\le \frac{4}{3}$ and $ N\in \mathbb{N}^{+}$. The main purpose of this section is to prove Proposition \ref{th11},  which  is concerned with   the optimal constant $D_{p,N}$ and the optimizers of \eqref{LTineq}.

We first note that the optimal constant  $D_{p,N}$ of \eqref{LTineq} is actually positive for the case $1<p\leq4/3$ and $N\in \mathbb{N}^+$.  Actually, one can recall from \cite{EHLRS,revis} the following classical fractional Lieb-Thirring inequality: for $1<p\leq4/3$,
\begin{equation*}\label{LT1}
\sum_{j=1}^{\infty}\big|\lam_{j}(\sqrt{-\Delta}-V)\big|^{\kappa}\le L_{\kappa}\int_{\R^{3}}V_{-}^{\kappa+3}dx, \  \   \forall\  V\in L^{\kappa+3}(\R^{3}, \R),
\end{equation*}
where  $\kappa:=\frac{3-2p}{p-1}>0,\ L_{\kappa}\in(0, \infty)$ denotes the best constant, and  $V_+(x):=\max\{V(x), 0\}$ is the positive part of $V$. Here $\lambda_j\left(\sqrt{-\Delta}-V\right)\leq0$ denotes the $j$th min-max level of the operator $\sqrt{-\Delta}-V$ on $L^2(\R^3,\C)$, which equals to the $j$th negative eigenvalue (counted with multiplicity) of the operator $\sqrt{-\Delta}-V$ on $L^2(\R^3,\C)$ if it exists, and vanishes otherwise. It then follows that  for any $N\in \mathbb{N}^+$, there exists an optimal constant $L_{\kappa,N}\in(0, \infty)$ such that
\begin{equation*}\label{LT2}
\sum_{j=1}^{N}\big|\lam_{j}(\sqrt{-\Delta}-V)\big|^{\kappa}\le L_{\kappa,N}\int_{\R^{3}}V_{-}^{\kappa+3}dx, \  \   \forall\  V\in L^{\kappa+3}(\R^{3}, \R).
\end{equation*}
By the same  dual argument of \cite[Lemma 5]{CMP},  one can check that
\begin{equation*}\label{indent}
L_{\kappa,N}D_{p,N}^{3}=\Big(\frac{\kappa}{\kappa+3}\Big)^{\kappa}\, \Big (\frac{3}{\kappa+3}\Big)^{3},
\end{equation*}
which therefore implies that $D_{p,N}>0$ holds for any $1<p\leq4/3$ and $N\in \mathbb{N}^+$.

In order to further study the dual fractional Lieb--Thirring inequality \eqref{LTineq}, we rewrite  it as the following equivalent minimization problem:
\begin{align}\label{def:eN}
D_{p,N}=\inf\limits_{\gamma\in \mathcal{R}_{N}}  \dfrac{\|\gamma\|_{\mathfrak{S}^{q}}^{\frac{3-2p}{3(p-1)}}\tr\big(\sqrt{-\Delta}\gamma\big)}
{\|\rho_{\gamma}\|_{p}^{\frac{p}{3(p-1)}}},\ \ 1<p\leq4/3,
\end{align}
where $\mathcal{R}_{N}$ and $\|\cdot\|_{\mathfrak{S}^{q}}$ are as in \eqref{RNdef} and \eqref{1.5}, respectively.
We begin with the following existence  and qualitative properties of
optimizers, which proves Proposition \ref{th11} (1).

\begin{lem}\label{lem2.1}
For any fixed  $1<p\leq\frac{4}{3}$ and $N\in\mathbb{N}^+$, the problem  $D_{p,N}$ defined by \eqref{def:eN}  admits at least one optimizer. Moreover,  any optimizer $\gamma$ of  $D_{p,N}$ satisfying $\|\gamma\|_{\mathfrak{S}^{q}}=1$ can be written as the form
\begin{equation}\label{gaform1}
\gamma=\sum_{j=1}^{R_{N}}n_{j}|w_{j}\rangle \langle w_{j}|,\ \,  \langle w_{j}, w_{k}\rangle=\delta_{jk},\ \ R_{N}\in[1,\,  N],
\end{equation}
where the constants  $q\in(0, +\infty]$  and  $n_j\in(0,1]$ are as in \eqref{q} and \eqref{njeq},  respectively, and the  orthonormal system  $(w_1, \cdots, w_{R_N})$ satisfies
\begin{align}\label{eq}
H^*_\gamma w_j:=
\Big[\sqrt{-\Delta}-\frac{pD_{p,N}^{3(p-1)}}{3(p-1) \big(\tr(\sqrt{-\Delta}\gamma)\big)^{3p-4}}\rho_{\gamma}^{p-1}\Big]w_{j}=\mu_jw_{j}\ \ \text{in}\ \ \R^3
\end{align}
\text{for\ some} $\mu_1\leq\mu_2\cdots\leq\mu_{R_N}<0$.
\end{lem}

\noindent \textbf{Proof.}
Using the IMS-type localization formula (cf. \cite{LiYau,ims2}),  the existence of optimizers  for $D_{p,N}$ can be proved in a similar way of \cite[Theorem 6 (i)]{CMP}, we thus omit the detailed proof for simplicity. We next prove \eqref{gaform1} by considering separately two different cases.

$Case\ 1: 1<p<\frac{4}{3}$.  Let $\gamma$ be an optimizer of $D_{p,N}$ satisfying $\|\gamma\|_{\mathfrak{S}^{q}}=1$, where $N\in\mathbb{N}^+$ and $q=\frac{3-2p}{4-3p}$.  For simplicity, we denote $\theta:=3(p-1)\in (0,1)$. It then follows from \eqref{1.5} and  (\ref{def:eN})  that
\begin{align}\label{w}
D_{p,N}^{3(p-1)}=\inf\limits_{\tilde{\gamma}\in\mathcal{R}_N}  \dfrac{(\tr\tilde{\gamma}^{q})^{1-\theta}\, \big(\tr(\sqrt{-\Delta}\tilde{\gamma})\big)^{\theta}}{\int_{\R^{3}}\rho_{\tilde{\gamma}}^{p}dx}.
\end{align}
Define
\begin{align}\label{gt1}
\gamma(t):=e^{itA}\gamma e^{-itA}, \ \, \text{where} \,\ t\in \R,\  A=|\varphi\rangle \langle \varphi|\  \, \text{and}\,  \ \varphi\in C_{0}^{\infty}(\R^{3}, \C).
\end{align}
It can be verified that $\gamma(t)\in \mathcal{R}_{N}$ and
\begin{align}\label{gt}
\gamma(t)=\gamma+ti(A\gamma -\gamma A)+o(t):=\gamma +t\delta+o(t)\  \ \ \text{as}\ \ t\to 0.
\end{align}
We thus deduce from \eqref{w} that for sufficiently  small $|t|>0$,
\begin{align}\label{a118}
&D_{p,N}^{3(p-1)}
\le\dfrac{(\tr|\gamma(t)|^{q})^{1-\theta}\big(\tr(\sqrt{-\Delta}\gamma(t))\big)^{\theta}}{\int_{\R^{3}}\rho_{\gamma(t)}^{p}dx}\nonumber\\[2mm]
=&\dfrac{\Big[1+qt\tr(\delta \gamma ^{q-1})+o(t)\Big]^{1-\theta}\Big[\tr(\sqrt{-\Delta}\gamma)+t\tr(\sqrt{-\Delta}\delta)+o(t)\Big]^{\theta}}{\int_{\R^{3}}\rho_{\gamma }^{p}dx+pt\int_{\R^{3}}\rho_{\delta}\rho_{\gamma}^{p-1}dx+o(t)}\\[2.5mm]
=&\dfrac{\Big[1+qt(1-\theta)\tr(\delta \gamma^{q-1})+o(t)\Big]\Big[\big(\tr(\sqrt{-\Delta}\gamma)\big)^{\theta}+t\theta\big(\tr(\sqrt{-\Delta}\gamma)\big)^{\theta-1}\tr(\sqrt{-\Delta}\delta)+o(t)\Big]}{\int_{\R^{3}}\rho_{\gamma }^{p}dx+pt\int_{\R^{3}}\rho_{\delta}\rho_{\gamma}^{p-1}dx+o(t)}\nonumber\\[2mm]
=&D_{p,N}^{3(p-1)}\bigg\{1+\frac{t\theta}{\tr(\sqrt{-\Delta}\gamma) }\tr\Big[\delta\Big(H^*_{\gamma}+\big(\tr(\sqrt{-\Delta}\gamma)\big)\frac{q(1-\theta)}{\theta}\gamma^{q-1}\Big)\Big]+o(t)\bigg\},\nonumber
\end{align}
where the operator $H^*_{\gamma}$ is defined by  (\ref{eq}).

Since (\ref{a118}) holds for sufficiently small $|t|>0$, we obtain from \eqref{gt} and  (\ref{a118})  that
\begin{align}\label{a120}
0&=\tr\Big[\delta\Big(H^*_{\gamma}+\big(\tr(\sqrt{-\Delta}\gamma)\big)\frac{q(1-\theta)}{\theta}\gamma^{q-1}\Big)\Big]\nonumber\\
&=\tr\Big[i(A\gamma-\gamma A)\Big(H^*_{\gamma}+\big(\tr(\sqrt{-\Delta}\gamma)\big)\frac{q(1-\theta)}{\theta}\gamma^{q-1}\Big)\Big]\\
&=\tr\Big[i(A\gamma -\gamma A)H^*_{\gamma }\Big]=\tr\Big[i(\gamma H^*_{\gamma }-H^*_{\gamma }\gamma )A\Big]\nonumber.
\end{align}
Recall from \eqref{gt1} that $A=|\varphi\rangle \langle \varphi|$, where $\varphi\in C_{0}^{\infty}(\R^{3}, \C)$ is arbitrary. This then implies from (\ref{a120}) that
\begin{align*}
\Big\langle\varphi, \ i\big(\gamma H^*_{\gamma }-H^*_{\gamma }\gamma \big)\varphi\Big\rangle=0,\ \, \forall\  \varphi\in L^{2}(\R^{3}, \C).
\end{align*}
Together with the self-adjointness of the operator $i(\gamma H^*_{\gamma }-H^*_{\gamma }\gamma )$, we further obtain that
\begin{align}\label{a121}
\gamma H^*_{\gamma }=H^*_{\gamma }\gamma \ \,  \text{on}\ \ L^{2}(\R^{3}, \C).
\end{align}
By the self-adjointness of the operators $\gamma$ and $H^*_{\gamma}$, we therefore deduce from \eqref{a121} that there exists $R_{N}\in[1,\,  N] $ such that
\begin{align}\label{2.9}
\gamma=\sum_{j=1}^{R_{N}}n_{j}|w_{j}\rangle \langle w_{j}|,\ \,  \langle w_{j}, w_{k}\rangle=\delta_{jk},
\end{align}
where $n_j>0$,  $w_j$ satisfies
\begin{align}\label{2.8}
H^*_{\gamma }w_{j}=\mu_jw_{j}\ \, \text{in}\ \ \R^3,\ \ j=1, \cdots, R_N,
\end{align}
and  $\mu_1\leq\mu_2\leq\cdots\le\mu_{R_N}\in \R.$ Note that $\|\gamma\|_{\mathfrak{S}^{q}}^q=\sum_{j=1}^{R_N}n_j^q=1$, which yields that $n_j\in(0, 1]$ for  $j=1, \cdots, R_N.$

Set $\gamma(t):=\gamma+t\delta$, where $|t|>0$ is sufficiently small, and  $\delta$ is a self-adjoint  operator  supported on the range of $\gamma$.  The same argument of  (\ref{a118}) then gives that
\begin{align*}
H^*_{\gamma }+\big(\tr(\sqrt{-\Delta}\gamma)\big) \frac{q(1-\theta)}{\theta}\gamma ^{q-1}= 0 \ \ \
\text{on}\    \text{span}\{w_{1}, \cdots, w_{R_{N}}\}.
\end{align*}
Together with \eqref{2.8},  we derive that
\begin{align*}
\mu_{j}+\big(\tr(\sqrt{-\Delta}\gamma)\big) \frac{q(1-\theta)}{\theta}n_{j}^{q-1}=0, \ \,  j=1,\cdots, R_{N},
\end{align*}
which  shows that
\begin{align*}
\mu_{j}<0\ \, \text{and} \ \, n_{j}=\Big(\frac{\theta}{q(1-\theta)\tr(\sqrt{-\Delta}\gamma)}\Big)^{\frac{1}{q-1}}|\mu_{j}|^{\frac{1}{q-1}}, \ \  \  j=1,\cdots, R_{N}.
\end{align*}
This therefore completes the proof of  \eqref{gaform1} for the case where $1<p<\frac{4}{3}$.

$Case\ 2: p=\frac{4}{3}$. In this case,  we suppose that  $\gamma$ is an optimizer of $D_{4/3,N}$  satisfying $\|\gamma\|_{\mathfrak{S}^{\infty}}=\|\gamma\|=1$. Letting $\gamma(t)$ be as in \eqref{gt1},  it then follows that  $\|\gamma(t)\|=\|\gamma\|=1$. Similar to \eqref{2.9} and \eqref{2.8},  one can deduce that there exists $R_{N}\in[1,\,  N] $ such that
$$
\gamma=\sum_{j=1}^{R_{N}}n_{j}|w_{j}\rangle \langle w_{j}|,\ \, \langle w_{j}, w_{k}\rangle=\delta_{jk},
$$
where $0<n_j\le 1$, and  $w_j$ satisfies
\begin{align}\label{2.8a}
\Big(\sqrt{-\Delta}-\frac{4}{3}D_{4/3, N}\rho_{\gamma}^{\frac{1}{3}}\Big)w_{j}=\mu_jw_{j}\ \ \text{in}\ \ \R^3,
\end{align}
for some $\mu_1\leq\mu_2\leq\cdots\le\mu_{R_N}\in \R.$

We next  prove  that $\mu_{j}<0$ and $n_j=\|\gamma\|=1$ hold for $j=1, \cdots,R_N$. Define
\begin{align*}
\gamma_t:=\gamma-t|w_{j}\rangle \langle w_{j}|\  \ \ \text{for}\ \ t\in(0,n_j),\ \, j=1, \cdots,R_N.
\end{align*}
One can calculate from \eqref{2.8a} that for $j=1, \cdots,R_N$,
\begin{align}\label{2.11}
D_{\frac{4}{3}, N}&\le 	\frac{\|\gamma_t\|^{\frac{1}{3}}\tr(\sqrt{-\Delta}\gamma_t)}{\int_{\R^{3}}\rho_{\gamma_t}^{4/3}dx} =\frac{\tr(\sqrt{-\Delta}\gamma_t)}{\int_{\R^{3}}\rho_{\gamma_t}^{4/3}dx}\nonumber\\[1mm]
&=\dfrac{\tr(\sqrt{-\Delta}\gamma)-t\big\langle w_{j},\sqrt{-\Delta}w_{j}\big\rangle }{\int_{\R^{3}}\big(\rho_{\gamma}-t|w_{j}|^{2}\big)^{4/3}dx}\\[1mm]
&=	D_{\frac{4}{3}, N}\dfrac{\inte\rho_\gamma^{4/3}dx-t\mu_{j}D^{-1}_{\frac{4}{3}, N}-\frac{4}{3}t\int_{\R^{3}}\rho^{1/3}_{\gamma}|w_{j}|^{2}dx}{\int_{\R^{3}}\big(\rho_{\gamma}-t|w_{j}|^{2}\big)^{4/3}dx}.\nonumber
\end{align}
Using the strict convexity of the function $t\mapsto t^{\frac{4}{3}}$, we thus obtain from \eqref{2.11} that for $j=1, \cdots,R_N$,
\begin{align*}
\mu_{j}\le -\frac{D_{4/3, N}}{t}\int_{\R^{3}}\Big[\big(\rho_{\gamma}-t|w_{j}|^{2}\big)^{\frac{4}{3}}-\rho_{\gamma}^{\frac{4}{3}}+\frac{4}{3}t\rho_{\gamma}^{\frac{1}{3}}|w_{j}|^{2}\Big]dx<0.
\end{align*}

By contradiction, we now assume that $0<n_{j}<\|\gamma\|=1$ holds for some  $j\in \{1,\cdots, R_{N}\}$.
Define $\gamma_t:=\gamma+t|w_{j}\rangle \langle w_{j}|$  for $t\in(0, 1-n_j)$. The same argument of \eqref{2.11} then yields that
\begin{align*}
D_{\frac{4}{3}, N}&\le \dfrac{\|\gamma_t\|^{\frac{1}{3}}\tr(\sqrt{-\Delta}\gamma_t)}{\int_{\R^{3}}\rho^{4/3}_{\gamma_t}dx}\\
&=D_{\frac{4}{3}, N}\dfrac{\inte\rho_\gamma^{4/3}dx+t\mu_{j}D^{-1}_{\frac{4}{3}, N}+\frac{4}{3}t\int_{\R^{3}}\rho^{1/3}_{\gamma}|w_{j}|^{2}dx}{\int_{\R^{3}}\big(\rho_{\gamma}+t|w_{j}|^{2}\big)^{4/3}dx}\\
&=D_{\frac{4}{3}, N}\dfrac{\inte\rho_\gamma^{4/3}dx+t\mu_{j}D^{-1}_{\frac{4}{3}, N}+\frac{4}{3}t\int_{\R^{3}}\rho^{1/3}_{\gamma}|w_{j}|^{2}dx}{\int_{\R^{3}}\rho_{\gamma}^{4/3}dx+\frac{4}{3}t\int_{\R^{3}}\rho^{1/3}_{\gamma}|w_{j}|^{2}dx+o(t)}\ \, \text{as}\ \ t\searrow0,
\end{align*}
which however contradicts with the fact $\mu_{j}<0$.  This therefore completes the proof of Lemma \ref{lem2.1}.
\qed

Employing Lemma \ref{lem2.1}, in this section we finally address the proof of Proposition \ref{th11}.

\vskip 0.05truein
\noindent \textbf{Proof of Proposition \ref{th11}.} (1). By Lemma \ref{lem2.1}, Proposition \ref{th11} (1) immediately holds true.

(2). Let $\gamma=\sum_{j=1}^{R_{N}}n_{j}|w_{j}\rangle \langle w_{j}|$ given by \eqref{gaform1} be an optimizer of the problem $D_{p,N}$, and suppose that $H^*_\gamma$ and  $\mu_1\leq\mu_2\cdots\leq\mu_{R_N}<0$ are as in \eqref{eq}, where $1<p\leq4/3$ and $N\in\mathbb{N}^+$.  Following Lemma \ref{lem2.1}, in order to complete the proof of Proposition \ref{th11}  (2),  it  suffices to prove that $\mu_1<\mu_2\leq\cdots\leq\mu_{R_N}<0$ are  the $R_{N}$ first eigenvalues (counted with multiplicity) of the operator $H^*_\gamma$ on $L^2(\R^3,\C)$, where the operator $H^*_{\gamma}$ is defined by  (\ref{eq}).

We first claim that $\mu_1\leq\mu_2\leq\cdots\leq\mu_{R_N}<0$ are  the $R_{N}$ first eigenvalues (counted with multiplicity) of the operator $H^*_\gamma$ on $L^2(\R^3,\C)$. On the contrary,  assume that there exist some $j_*\in\{1, \cdots,R_N\}$ and an eigenpair $(\mu, \, w)$  of $H^*_\gamma$ such that $\mu<\mu_{j_*}<0$, $\|w\|_2$=1 and $\langle w,\, w_j\rangle=0$ holds for all $j=1, \cdots,R_N$.
Define
\begin{align*}
\tilde{\gamma}:=
\gamma+n_{j_*}|w\rangle \langle w|-n_{j_*}|w_{j_*}\rangle \langle w_{j_*}|.
\end{align*}
By  the  convexity of the function $t\mapsto t^{p}$ for $1<p\leq4/3$, we then calculate from \eqref{eq} that
\begin{align*}
\int_{\R^{3}}\rho_{\tilde{\gamma}}^{p}dx\ge \inte\rho_\gamma^p dx+pn_{j_*}\int_{\R^{3}}\rho_{\gamma}^{p-1}\big(|w|^{2}-|w_{j_*}|^{2}\big)dx,
\end{align*}
and
\begin{align*}
\tr(\sqrt{-\Delta} \tilde{\gamma})
&=\tr(\sqrt{-\Delta} \gamma)+n_{j_*}\big\langle w, \sqrt{-\Delta}w\big\rangle-n_{j_*}\big\langle w_{j_*}, \sqrt{-\Delta}w_{j_*} \big\rangle\nonumber\\[2mm]
&=\tr(\sqrt{-\Delta} \gamma)+pn_{j_*}C(p,N)\int_{\R^{3}}\rho_{\gamma}^{p-1}(|w|^{2}-|w_{j_*}|^{2})dx+n_{j_*}(\mu-\mu_{j_*})\nonumber\\	 &<\tr(\sqrt{-\Delta} \gamma)+pn_{j_*}C(p,N)\int_{\R^{3}}\rho_{\gamma}^{p-1}\big(|w|^{2}-|w_{j_*}|^{2}\big)dx,
\end{align*}
where $C(p,N)=\frac{D_{p,N}^{3(p-1)}}{3(p-1) \big(\tr(\sqrt{-\Delta}\gamma)\big)^{3p-4}}>0$.
This then gives that
\begin{align*}
D_{p,N}^{3(p-1)}&\le \dfrac{\big(\tr(\sqrt{-\Delta}\tilde{\gamma})\big)^{3(p-1)}}{\int_{\R^{3}}\rho_{\tilde{\gamma}}^{p}dx}\\
&<\dfrac{\Big(\tr(\sqrt{-\Delta} \gamma)+{p}n_{j_*}C(p,N)\int_{\R^{3}}\rho_{\gamma}^{p-1}\big({|w|^{2}-|w_{j_*}|^{2}}\big)dx\Big)^{3(p-1)}}
{\inte\rho_\gamma^p dx+pn_{j_*}\int_{\R^{3}}\rho_{\gamma}^{p-1}\big(|w|^{2}-|w_{j_*}|^{2}\big)dx}\\
&=D_{p,N}^{3(p-1)}\dfrac{\Big(1+pn_{j_*}\frac{C(p,N)}{\tr\sqrt{-\Delta} \gamma}\int_{\R^{3}}\rho_{\gamma}^{p-1}\big({|w|^{2}-|w_{j_*}|^{2}}\big)dx\Big)^{3(p-1)}}{1+3p(p-1) n_{j_*}\frac{C(p,N)}{\tr\sqrt{-\Delta} \gamma}\int_{\R^{3}}\rho_{\gamma}^{p-1}\big(|w|^{2}-|w_{j_*}|^{2}\big)dx}
\le D_{p,N}^{3(p-1)},
\end{align*}
a contradiction.  This proves the claim that $\mu_1\leq\mu_2\leq\cdots\leq\mu_{R_N}<0$ are  the $R_{N}$ first eigenvalues (counted with multiplicity) of the operator $H^*_\gamma$ on $L^2(\R^3,\C)$.

We now prove that  $\mu_1<\mu_2$ holds true.  Recall from (\ref{eq}) that  
\begin{align}\label{a127}
\big(\sqrt{-\Delta}\, -\mu_{j}\big)w_{j}=pC(p,N)\rho_{\gamma}^{p-1}w_{j}\ \, \text{in}\ \ \R^{3},\ \,  j=1, \cdots, R_{N},
\end{align}
where the constant $C(p, N)=\frac{D_{p,N}^{3(p-1)}}{3(p-1) \big(\tr(\sqrt{-\Delta}\gamma)\big)^{3p-4}}>0$. This yields that
\begin{align}\label{119}
w_{j}(x)=pC(p,N)\Big(G_{j}*\big(\rho_{\gamma}^{p-1}w_{j}\big)\Big)(x)\ \  \text{in}\ \ \R^{3},
\end{align}
where $G_{j}(x)$ is the Green's function of the operator $\sqrt{-\Delta}-\mu_{j}$ in $\R^{3}$.
The same argument of \cite[Lemma A.1]{cb3} gives from \eqref{a127}  that $(w_{1}, \cdots, w_{R_{N}})\in \big(L^{r}(\R^{3}, \C)\big)^{R_{N}}$  holds for any  $r\in [2, \infty)$, and hence there exists a constant $r>3$ such that
\begin{align}\label{2.17}
\rho_{\gamma}^{p-1}w_{j}\in L^{r}(\R^{3}, \C), \ \    j=1, \cdots, R_{N}.
\end{align}
Note from \cite[Lemma C.1]{CPAM} that
\begin{align}\label{a137}
G_{j}\in L^{\frac{r}{r-1}}(\R^{3}, \R), \ \    j=1, \cdots, R_{N}.
\end{align}
Applying  \cite[Lemma A.2]{YNTS}, we thus deduce from (\ref{119})--(\ref{a137}) that
\begin{align}\label{120}
w_{j}\in C(\R^{3}, \C)\ \, \text{and}\ \,  \lim\limits_{|x|\to\infty}|w_{j}(x)|= 0, \ \, j=1, \cdots, R_{N}.
\end{align}
Since $\mu_{1}<0$ is the first eigenvalue of the operator  $\sqrt{-\Delta}-pC(p,N)\rho_{\gamma}^{p-1}$ on $L^2(\R^3,\C)$, the same argument of  \cite[Section 11.8]{Analysis} yields from \eqref{a127} and \eqref{120} that  $\mu_1<\mu_2$ holds true,
$w_1$  is unique (up to a constant phase) and can be chosen to be a strictly positive function.
This therefore completes the proof of  Proposition \ref{th11} (1) and (2).

(3). Let $\gamma=\sum_{j=1}^{R_{N}}n_{j}|w_{j}\rangle \langle w_{j}|$  given by \eqref{gaform} be an optimizer of the problem $D_{p,N}$.   Using a similar argument of \cite[Lemma C.2]{CPAM}, we conclude from \eqref{Qeq} and \eqref{120}   that  there exist constants $C>0$ and $\tilde{C}_j>0$ such that
\begin{align}\label{a15}
|w_{j}(x)|\le C\big(1+|x|^{4}\big)^{-1}\  \ \  \text{in}\  \,  \R^{3}, \ \     j=1, \cdots, R_{N},
\end{align}
and
\begin{align}\label{low}
\lim\limits_{|x|\to\infty}|x|^{4}w_{j}(x)=\tilde{C}_j\inte \rho_{\gamma}^{p-1}(y)w_{j}(y)dy, \  \   j=1, \cdots, R_{N}.
\end{align}
Since  $w_1$ can be chosen as a strictly positive function, it yields from \eqref{120} and \eqref{low} that there exists $C_1>0$ such that
\begin{align}\label{a17}
|w_{1}(x)|\geq C_1\big(1+|x|^{4}\big)^{-1}\  \ \,  \text{in}\  \  \R^{3}.
\end{align}

We now claim that
\begin{align}\label{cliam129}
\text{if}\ D_{p, N}\ \text{admits an optimizer}\ \gamma\ \text{of rank}\  N, \ \text{then}\ D_{p, 2N}<D_{p, N} \ \text{holds true}.
\end{align}
Suppose that
$$
\gamma=\sum\limits_{j=1}^{N}n_{j}|w_{j}\rangle \langle w_{j}| \ \text{given by }\eqref{gaform} \ \text{is an optimizer of the problem}\ D_{p,N}.
$$
Then  the orthonormal system $(w_1, \cdots, w_N)$ satisfies the estimates \eqref{a15} and \eqref{a17}.
Set $w_{jR}(x):=w_{j}(x-Re_{1})$  for $R>0$, and define the  Gram matrix  $G_{R}$ as
\begin{equation}\label{3.6b}
\begin{split}
G_R:=\begin{bmatrix}
\mathbb{I}_{N}& E_{R}\\
E_{R}^{*}&\mathbb{I}_{N}
\end{bmatrix}=:
\left(
\begin{split}
&w_1\\[-2mm]
&\ \vdots\\[-2mm]
&w_N\\
&w_{1R}\\[-2mm]
&\ \vdots\\[-2mm]
&w_{NR}
\end{split}
\right)
\big(w_1, \cdots,w_N,w_{1R},\cdots, w_{NR}\big)
,
\end{split}
\end{equation}
where $e_1=(1,0,0)$, $E_{R}=\big(e_{ij}^{R}\big)_{i,j=1}^N$ and $e_{ij}^{R}= \langle w_{i}, w_{j R}\rangle$. Since $\lim\limits_{R\to\infty}|e_{ij}^{R}|=0$,
we obtain that  $G_{R}$ is positive definite for sufficiently large $R>0$.  Set
\begin{align*}
(\tilde{w}_{1,R}, \cdots, \tilde{w}_{2N,R}):=(w_{1}, \cdots, w_{N}, w_{1R}, \cdots, w_{N R})G_{R}^{-\frac{1}{2}},
\end{align*}
and
\begin{align}\label{gar}
\tilde{\gamma}_{R}:=\sum\limits_{i=1}^{N}n_i|\tilde{w}_{i,R}\rangle \langle \tilde{w}_{i,R}|+\sum\limits_{i=1}^{N}n_i|\tilde{w}_{i+N, R}\rangle \langle \tilde{w}_{i+N, R}|.
\end{align}
It then follows that  $(\tilde{w}_{1,R}, \cdots, \tilde{w}_{2N,R})$ is an orthonormal system in $L^2(\R^3,\C)$, and thus
\begin{align*}
\|\tilde{\gamma}_R\|_{\mathfrak{S}^{q}}=
\left\{\begin{array}{lll}
\big(2\sum\limits_{j=1}^{N}n_j^q\big)^{1/q}=2^{1/q}\|\gamma\|_{\mathfrak{S}^{q}}=2^{1/q},\ \ &\text{if}\ \ 1\le q<\infty; \\[2mm]
\|\gamma\|=1,\ \ &\text{if}\ \ q=\infty,
\end{array}\right.
\end{align*}
where $q$ is as in \eqref{q}.

Note that
\begin{align*}
G_{R}^{-\frac{1}{2}}=\begin{bmatrix}
\mathbb{I}_{N}& 0\\
0&\mathbb{I}_{N}
\end{bmatrix}-\frac{1}{2}\begin{bmatrix}
0& E_{R}\\
E_{R}^{*}&0
\end{bmatrix}+O(e^{2}_{R})\ \  \text{as}\ \ R\to\infty,
\end{align*}
where
\begin{align*}
e_{R}:=\max_{i,j}|\langle w_i, w_{jR}\rangle|=o(1)\ \  \text{as}\ \ R\to\infty.
\end{align*}
This yields from \eqref{gar} that
\begin{align}\label{2.15}
\tilde{\gamma}_R
=&\ \gamma+\gamma_R-\frac{1}{2}\sum_{i=1}^{N}\sum_{j=1}^{N}n_i\big(|w_i\rangle\langle \overline{e^R_{ij}}w_{jR}|+|\overline{e^R_{ij}}w_{jR}\rangle\langle w_i|\big)\\
&-\frac{1}{2}\sum_{i=1}^{N}\sum_{j=1}^{N}n_i\big(|e^R_{ji}w_{j}\rangle\langle w_{iR}|+|w_{iR}\rangle\langle e^R_{ji}w_{j}|\big)+O(e_R^2) \ \  \, \mbox{as} \ \ R\to\infty,\nonumber
\end{align}
where   $\gamma_R:=\sum_{j=1}^{N}n_j|w_{jR}\rangle\langle w_{jR}|$. As a consequence of \eqref{Qeq} and \eqref{a15}, one can  calculate from \eqref{2.15} that
\begin{align}\label{2.27}
D_{p,2N}^{3(p-1)}&\le \dfrac{\|\tilde{\gamma}_{R}\|_{\mathfrak{S}^{q}}^{q(4-3p)}\big(\tr(\sqrt{-\Delta}\tilde{\gamma}_{R})\big)^{3(p-1)}}{\int_{\R^{3}}\rho_{\tilde{\gamma}_{R}}^{p}dx}
=\dfrac{2^{4-3p}\big(2\tr (\sqrt{-\Delta}\gamma)+O(e^{2}_{R})\big)^{3(p-1)}}{2\int_{\R^{3}}\rho_{\gamma}^{p}dx+\int_{\R^{3}}\big(\rho_{\tilde{\gamma}_{R}}^{p}-\rho_{\gamma}^{p}-\rho_{\gamma_R}^{p}\big)dx}\nonumber\\
&=D_{p,N}^{3(p-1)} \Big[1-\frac{1}{2\inte\rho^p_\gamma dx}\int_{\R^{3}}\big(\rho_{\tilde{\gamma}_{R}}^{p}-\rho_{\gamma}^{p}-\rho_{\gamma_R}^{p}\big)dx+O(e^{2}_{R})\Big]\nonumber\\
&=D_{p,N}^{3(p-1)} \Big[1-\frac{1}{2\inte\rho^p_\gamma dx}\int_{\R^{3}}\big((\rho_\gamma+\rho_{\gamma_{R}})^{p}-\rho_{\gamma}^{p}-\rho_{\gamma_R}^{p}\big)dx\Big.\\
&\ \, \ \ \ \ \ \ \ \ \ \ \ \ \Big.+O(e^{2}_{R})\Big]\ \ \text{as}\ \ R\to\infty.\nonumber
\end{align}

Since it follows  from \eqref{a17} that 
\begin{align*}
\rho_{\gamma}(x)=\sum_{i=1}^{N}|w_{i}(x)|^{2}\ge |w_{1}(x)|^{2}\ge C^2_{1}(1+|x|^{4})^{-2}\ge C^2_{1}(1+16R^{4})^{-2}\ \, \text{in}\ B_R,
\end{align*}
and
\begin{align*}
\rho_{\tilde{\gamma}_{R}}(x)=\sum_{i=1}^{N}|w_{i}(x-Re_1)|^{2}\ge C^2_{1}(1+|x-Re_{1}|^{4})^{-2}\ge C^2_{1}(1+16R^{4})^{-2}\ \, \text{in}\ B_R,
\end{align*}
one gets that for sufficiently large $R>0$,
\begin{equation}\label{2.28}
\begin{split}
&\int_{\R^{3}}\Big((\rho_{\gamma}+\rho_{\tilde{\gamma}_{R}})^{p}-\rho_{\gamma}^{p}-\rho_{\tilde{\gamma}_{R}}^{p}\Big)dx\\
\ge&\frac{4\pi R^3}{3} C_{1}^{2p}(2^{p}-2)(1+16R^{4})^{-2p}\geq \tilde{C}_1R^{-8p+3},
\end{split}
\end{equation}
where $\tilde{C}_1>0$ is independent of $R>0$, and we have used the fact that the function $(x,y)\to (x+y)^{p}-x^{p}-y^{p}$ increases in $x$ and $y$ separately for $1<p\le \frac{4}{3}$.
Moreover, it can be verified from \eqref{a15} that
\begin{align*}
e_{R}&\leq\max_{i,j}\int_{\R^{3}}|w_{i}(x)||w_{j}(x-Re_{1})|dx\\
&\le C^2\int_{\R^{3}}\frac{1}{(1+|x|^{4})(1+|x-Re_{1}|^{4})}dx\\
&\leq C^2\big(1+\frac{R^{4}}{16}\big)^{-1}\int_{B_{R/2}}\frac{1}{(1+|x|^{4})} dx+C^2\big(1+\frac{R^{4}}{16}\big)^{-1}\int_{B^c_{R/2}} \frac{1}{(1+|x-Re_{1}|^{4})}dx\\
&\le C_2R^{-4},
\end{align*}
where $C_2>0$ is independent of $R>0$. Together with \eqref{2.28}, we thus deduce from \eqref{2.27} that  for sufficiently large $R>0$,
\begin{align*}
D_{p, 2N}^{3(p-1)}&\leq D_{p,N}^{3(p-1)} \Big(1-C'_1R^{-8p+3}+C'_2R^{-8}\Big)<D_{p, 2N}^{3(p-1)},
\end{align*}
due to the assumption $1<p\leq 4/3$. This proves the claim (\ref{cliam129}).

Since any optimizer of  $D_{p, 1}$ is of rank $1$, we deduce  from (\ref{cliam129}) that $D_{p,2}<D_{p,1}$ holds true, which  further implies that
\begin{align*}
\text{any optimizer of} \ D_{p, 2}\ \text{is of rank}\ 2.
\end{align*}
Together with \eqref{cliam129},  we hence  obtain that $D_{p,4}<D_{p,2}$ holds true.  If $D_{p, 3}<D_{p, 2}$, then we take $N_{3}=3$,  otherwise we take $N_{3}=4$.  One can check that
\begin{align*}
\text{any optimizer of} \ D_{p, N_3}\ \text{is of rank}\ N_3.
\end{align*}
Consequently, we can prove by induction that there exists an infinite sequence of integers $N_{1}=1<N_{2}=2<N_{3}<\cdots$ such that any optimizer of $D_{p, N_k}$ is of rank $ N_k$. This completes  the proof of Proposition \ref{th11}.  \qed

\section{Existence and Non-existence of Minimizers}
In this section, we complete the  proof of Theorem \ref{th21} on the existence and nonexistence of minimizers for the problem  $E_{a}(2)$ defined by (\ref{pro11}).

We first introduce the following minimization problem:
\begin{align}\label{prolamdf}
E_{a}(\lam):=\inf\limits_{\gamma\in\mathcal{P}_{\lam}}\mathcal{E}_{a}(\gamma),\ \, a>0,\ \lambda>0,
\end{align}
where  the energy functional $\mathcal{E}_{a}(\gamma)$ is as in  (\ref{EVdf}),
\begin{align}\label{pn}
\mathcal{P}_{\lam}:=\Big\{\gamma:\,  \gamma=&\sum_{i=1}^{N}|u_{i}\rangle \langle u_{i}|+(\lam-N)|u_{N}\rangle \langle u_{N}| :\  u_{i}\in H^{\frac{1}{2}}(\R^{3}, \C),\nonumber\\  &\langle u_{i},u_{j}\rangle=\delta_{ij}, \  i,j=1,\cdot\cdot\cdot, N
\Big\},
\end{align}
and $N\in\mathbb{N}^+$ is the smallest integer such that $0<\lam\le N$.
The following lemma provides some analytical properties of the energy $E_{a}(\lambda)$.

\begin{lem}\label{lem32}
For any fixed $N\in \mathbb{N}^{+}$, let $D_{4/3, N}\in (0, \infty)$  be  defined by  $(\ref{LTineq})$.  Then the problem $E_{a}(\lam)$ defined in $(\ref{prolamdf})$ satisfies the following properties:
\begin{enumerate}
\item [(1)] If $0< a \le D_{4/3, N}$, then $-\infty<E_{a}(\lam)< 0$ holds for any $\lam\in (0,N]$.  In particular, we have $E_{D_{4/3, 2}}(2)=-2m$,  and $E_{D_{4/3, 2}}(2)$ has no minimizers.
\item  [(2)]  If $0< a \le D_{4/3, N}$, then
\begin{equation*}\label{subad1}
E_{a}(\lam)\le E_{a}(\lam_{1}) +E_{a}(\lam-\lam_{1}),\ \, \forall\ 0<\lam_{1}<\lam\le N.
\end{equation*}


\item [(3)]  If $a> D_{4/3, N}$, then $E_{a}(N)=-\infty$.
\end{enumerate}
\end{lem}

\noindent \textbf{Proof.} (1). Let  $N\in \mathbb{N}^{+}$ be fixed, and take $\lam\in (0,N]$. By the definition of $D_{4/3, N}$, one can verify that  for any  $\gamma\in \mathcal{P}_\lambda$,
\begin{align}\label{lowerbd}
\mathcal{E}_{a}(\gamma)&=\tr(\sqrt{-\Delta+m^{2}}-m)\gamma-a\int_{\R^{3}}\rho_{\gamma}^{\frac{4}{3}}dx\nonumber\\
&\ge \tr\big(\sqrt{-\Delta}\gamma\big) -m\lam-a\int_{\R^{3}}\rho_{\gamma}^{\frac{4}{3}}dx\\
&\ge(1-\frac{a}{D_{4/3, N}})\tr\big(\sqrt{-\Delta}\gamma\big) -m\lam\ge -m\lam,\ \, \forall\ 0<a\leq D_{4/3, N}.\nonumber
\end{align}
This yields that if $0< a \le D_{4/3, N}$, then $E_{a}(\lam)$ is finite for any $\lam\in (0,N]$. Let $\gamma\in\mathcal{P}_\lambda$
and define $\gamma_t(x,y):=t^3\gamma(tx,ty)$ for $t>0$.
By  the  operator inequality \begin{equation*}\label{3.8c}
\sqrt{-\Delta+m^2}-m\leq \frac{-\Delta}{2m}, \ \, \forall\ m>0,
\end{equation*}
we obtain that
\begin{align*}
\mathcal{E}_{a}(\gamma_{t})&\le \frac{1}{2m}\tr\big(-\Delta\gamma_{t}\big)-a\int_{\R^{3}}\rho_{\gamma_{t}}^{\frac{4}{3}}dx \\ &=\frac{t^2}{2m}\tr\big(-\Delta\gamma\big)-at\int_{\R^{3}}\rho_{\gamma}^{\frac{4}{3}}dx\\
&<0,\ \, \text{if}\ t>0\ \text{is sufficiently small},
\end{align*}
which thus implies that $E_{a}(\lam)< 0$ holds for any $a>0$.

By Proposition \ref{th11} (3), we can choose $\gamma^{(2)}=\sum_{i=1}^2|w_i\rangle\langle w_i|\in \mathcal{P}_2$ as an optimizer of $D_{4/3, 2}$.  Denoting $\gamma_{t}^{(2)}(x, y):=t^3\gamma^{(2)}(tx, ty)$,  the dominated convergence theorem yields that
\begin{align}\label{3.5}
-2m &\le E_{D_{4/3, 2}}(2)\le \lim\limits_{t\to \infty}\mathcal{E}_{D_{4/3, 2}}(\gamma_{t}^{(2)})\nonumber\\
&=\lim\limits_{t\to \infty}\Big[\mathcal{E}_{D_{4/3, 2}}(\gamma_{t}^{(2)})-\tr\big(\sqrt{-\Delta}\gamma_{t}^{(2)}\big)+D_{4/3, 2}\int_{\R^{3}}\rho_{\gamma_{t}^{(2)}}^{\frac{4}{3}}dx\Big]\nonumber\\
&=-2m+\lim\limits_{t\to \infty}\Big[\tr\big(\sqrt{-\Delta+m^{2}}-\sqrt{-\Delta}\, \big)\gamma_{t}^{(2)}\Big] \\
&=-2m+\lim\limits_{t\to \infty}\sum_{i=1}^{2}\int_{\R^{3}}\big(\sqrt{t^{2}|\xi|^{2}+m^{2}}-\sqrt{t^{2}|\xi|^{2}}\, \big)|\hat{w}_{i}|^{2}d\xi\nonumber\\
&=-2m+\lim\limits_{t\to \infty}\sum_{i=1}^{2}\int_{\R^{3}}\frac{m^2}{\sqrt{t^{2}|\xi|^{2}+m^{2}}+\sqrt{t^{2}|\xi|^{2}}}|\hat{w}_{i}|^{2}d\xi\nonumber\\
&=-2m.\nonumber
\end{align}
Therefore, if $E_{D_{4/3, 2}}(2)$ has a minimizer $\gamma$, then the definition of $D_{4/3, 2}$ gives from \eqref{3.5} that
\begin{align*}
\tr\big(\sqrt{-\Delta}\gamma\big) < \tr\big(\sqrt{-\Delta+m^{2}}\,  \gamma\big) =D_{4/3, 2}\int_{\R^{3}}\rho_{\gamma}^{\frac{4}{3}}dx\le \tr\big(\sqrt{-\Delta} \gamma\big),
\end{align*}
a contradiction. This proves that $E_{D_{4/3, 2}}(2)$ has no minimizers.

(2). For any given $0< \lam_{1}<\lam\le N$, set
\begin{align}\label{ga1}
\gamma_{1}=\sum_{i=1}^{N_{1}}|\varphi_{i}\rangle \langle \varphi_{i}|+(\lam_{1}-N_{1})|\varphi_{N_{1}}\rangle \langle \varphi_{N_{1}}|\in \mathcal{P}_{\lam_{1}},
\end{align}
and
\begin{align*}\label{ga2}
\gamma_{2}=\sum_{j=1}^{N_{2}}|\psi_{j}\rangle \langle \psi_{j}|+(\lam-\lam_{1}-N_{2})|\psi_{N_{2}}\rangle \langle \psi_{N_{2}}|\in \mathcal{P}_{\lam-\lam_{1}},
\end{align*}
where $N_{1}, N_{2}\in \mathbb{N}^{+}$ are the smallest integers such that $\lam_{1}\le N_{1}$ and $\lam-\lam_{1}\le N_{2}$, respectively.   Define
\begin{equation*}
\mathcal{K}_{\lam}:=\big\{ \gamma{'}\in\mathcal{B}\big(L^{2}(\R^{3}, \C)\big): \ 0\le \gamma{'}=(\gamma{'})^{*}\le 1,\  \mathrm{Tr}(\gamma)=\lambda,\  \mathrm{Tr}(\sqrt{-\Delta}\gamma)<\infty\big\},
\end{equation*}
where $\mathcal{B}\big(L^{2}(\R^{3}, \C)\big)$ denotes the set of bounded linear operators on $L^{2}(\R^{3}, \C)$.  The same argument of \cite[Lemma 11]{Arma} then gives that
\begin{align}\label{equver}
E_{a}(\lam)=\inf\limits_{\gamma{'}\in \mathcal{K}_{\lam}} \mathcal{E}_{a}(\gamma{'}).
\end{align}

Denote  $\psi_{j}^{\tau}(x):=\psi_{j}(x-\tau e_{1})$ for $\tau>0$,
where  $e_{1}=(1,0,0)$ and $j=1, \cdots, N_2$. Similar to (\ref{2.15}),  using the system $(\varphi_{1}, \cdots, \varphi_{N_1},\psi_{1}^{\tau}, \cdots, \psi_{N_2}^{\tau})$,  we can define an operator $\gamma_\tau\in\mathcal{K}_\lambda$ as follows:
\begin{align}\label{gatau1}
\gamma_{\tau}=&\, \gamma_{1}+\gamma_{2}^{\tau}-\sum_{i=1}^{N_{1}}\sum_{j=1}^{N_{2}}(|\varphi_{i}\rangle \langle \overline{e_{ij}^{\tau}}\psi_{j}^{\tau}|+|\overline{e_{ij}^{\tau}}\psi_{j}^{\tau}\rangle \langle \varphi_{i}|)\nonumber\\
&\quad -\frac{1}{2}(\lam_{1}-N_{1})\sum_{j=1}^{N_{2}}\big(|\varphi_{N_{1}}\rangle \langle \overline{e^{\tau}_{N_{1}j}}\psi_{j}^{\tau}|+|\overline{e^{\tau}_{N_{1}j}}\psi_{j}^{\tau}\rangle \langle \varphi_{N_{1}}|\big)\\
&\quad -\frac{1}{2}(\lam-\lam_{1}-N_{2})\sum_{i=1}^{N_{1}}\big(|\varphi_{i}\rangle \langle \overline{e^{\tau}_{iN_{2}}}\psi_{N_{2}}^{\tau}|+|\overline{e^{\tau}_{iN_{2}}}\psi_{N_{2}}^{\tau}\rangle \langle \varphi_{i}|\big)\nonumber\\
&\quad+O(e_{\tau}^{2})\ \, \text{as}\ \ \tau\to \infty,\nonumber
\end{align}
where $e^{\tau}_{ij}:=\langle\varphi_{i}, \psi_{j}^{\tau}\rangle,\ e_{\tau}:=\max\limits_{i,j}|\langle\varphi_{i}, \psi_{j}^{\tau}\rangle|=o(1)$ as $\tau\to \infty$, $\gamma_{1}$ is as in (\ref{ga1}) and
\begin{align}\label{c317}
\gamma_{2}^{\tau}=\sum_{j=1}^{N_{2}}|\psi^{\tau}_{j}\rangle \langle \psi^{\tau}_{j}|+(\lam-\lam_{1}-N_{2})|\psi^{\tau}_{N_{2}}\rangle \langle \psi^{\tau}_{N_{2}}|.
\end{align}
One can check from (\ref{gatau1}) and (\ref{c317}) that
\begin{align}\label{trgatau}
&\quad\tr(\sqrt{-\Delta+m^{2}}-m)\gamma_{\tau}\nonumber\\
&=\tr(\sqrt{-\Delta+m^{2}}-m)\gamma_{1}+\tr(\sqrt{-\Delta+m^{2}}-m)\gamma_{2}+o(1)\ \ \text{as}\ \ \tau\to \infty,
\end{align}
and
\begin{align}\label{rhoL1}
\int_{\R^{3}}|\rho_{\gamma_{\tau}}-\rho_{\gamma_{1}}-\rho_{\gamma_{2}^\tau}|dx=o(1)\ \ \text{as}\ \ \tau\to \infty,
\end{align}
where $\rho_{\gamma_{2}^{\tau}}(x)=\rho_{\gamma_{2}}(x-\tau e_{1})$.  Applying the following inequality (cf. \cite[Lemma 2.1]{Duke})
\begin{equation*}\label{sqineq}
\big(\sqrt{\rho_{\gamma}},\,  \sqrt{-\Delta}\sqrt{\rho_{\gamma}}\big)\le \tr(\sqrt{-\Delta}\gamma),
\end{equation*}
we then  deduce from (\ref{trgatau}) that $\{\sqrt{\rho_{\gamma_{\tau}}}\}$ is bounded uniformly in $H^{\frac{1}{2}}(\R^{3})$, and thus
$\{\rho_{\gamma_{\tau}}\}$ is bounded uniformly in $L^{r}(\R^{3},\C)$ for  any $1\le r\le \frac{3}{2}$.
By the interpolation inequality, we thus conclude from (\ref{rhoL1}) that
\begin{align*}
\rho_{\gamma_{\tau}}(x)-\rho_{\gamma_{1}}(x)-\rho_{\gamma_{2}}(x-\tau e_{1})\to 0 \ \ \text{strongly in}\ \ L^{r}(\R^{3})\ \  \text{as}\ \ \tau \to \infty,\ \ \forall\ r\in [1, 3/2) ,
\end{align*}
which further implies that
\begin{align}\label{rhogatauL}
\lim\limits_{\tau \to \infty}\int_{\R^{3}}\rho_{\gamma_{\tau}}^{\frac{4}{3}}dx&=\lim\limits_{\tau\to \infty}\int_{\R^{3}}\Big(\rho_{\gamma_{1}}(x)+\rho_{\gamma_{2}}(x-\tau e_{1})\Big)^{\frac{4}{3}}dx\nonumber\\
&=\int_{\R^{3}}\Big(\rho_{\gamma_{1}}^{\frac{4}{3}}+\rho_{\gamma_{2}}^{\frac{4}{3}}\Big)dx.
\end{align}

As a consequence of \eqref{equver} and (\ref{trgatau}), we obtain from (\ref{rhogatauL}) that
\begin{align}\label{3.14}
E_{a}(\lam)&=\inf\limits_{\gamma{'}\in \mathcal{K}_{\lam}} \mathcal{E}_{a}(\gamma{'})\le \lim\limits_{\tau\to \infty}\mathcal{E}_{a}(\gamma_{\tau})=\mathcal{E}_{a}(\gamma_{1})+\mathcal{E}_{a}(\gamma_{2}).
\end{align}
Since $\gamma_{1}\in\mathcal{P}_{\lambda_1}$ and $\gamma_{2}\in\mathcal{P}_{\lambda-\lambda_1}$ are arbitrary, 
we  get from \eqref{3.14} that
\begin{equation*}\label{subad}
E_{a}(\lam)\le E_{a}(\lam_{1}) +E_{a}(\lam-\lam_{1}), \ \, \forall\ 0< \lam_{1}<\lam\le N,
\end{equation*}
which  therefore completes the proof of Lemma  \ref{lem32} (2).


(3).  Let $0< a \le D_{4/3,N}$. By Lemma  \ref{lem32} (1) and (2),  we have
$$E_{a}(\lam)\le E_{a}(\lam_{1})+E_{a}(\lambda-\lam_{1})<E_{a}(\lam_{1}),\ \, \forall\ 0< \lam_{1}<\lam\le N.
$$
This shows that if $0< a \le D_{4/3, N}$, then $E_{a}(\lam)$ decreases strictly  in $\lam\in (0,N]$. Following Proposition \ref{th11}, we can take $\gamma=\sum_{i=1}^{R_{N}}|w_{i}\rangle \langle w_{i}|\in\mathcal{P}_{R_N}$ as an optimizer of $D_{4/3, N}$, where $N\geq R_{N}\in\mathbb{N}^+$. Set $\gamma_t(x,y):=t^3\gamma(tx,ty)$ for $t>0$. We thus calculate that for $ a>D_{4/3,N},$
\begin{align*}
E_{a}(N)&\le E_{a}(R_{N})\le \lim\limits_{t\to \infty}\mathcal{E}_{a}(\gamma_{t})\nonumber\\
&= \lim\limits_{t\to \infty} t\Big(\tr\big(\sqrt{-\Delta}\gamma\big)-a\int_{\R^{3}}\rho_{\gamma}^{\frac{4}{3}}dx\Big)\\
&=\lim\limits_{t\to \infty} t\Big(1-\frac{a}{D_{4/3,N}}\Big)\tr\big(\sqrt{-\Delta}\gamma\big)
=-\infty,
\end{align*}
which then completes the proof of Lemma \ref{lem32}.\qed

\vspace{.05cm}

Employing Lemma \ref{lem32}, we now analyze the following fundamental properties of minimizers for  $E_{a}(\lam)$.

\begin{lem}\label{prop34}
For any fixed $N\in\mathbb{N}^+$, let $E_{a}(\lam)$ be defined by $(\ref{prolamdf})$, where $a\in(0, D_{4/3, N}]$ and  $\lam\in(0, N]$. Suppose $\gamma\in\mathcal{P}_\lambda$ is a minimizer of $E_{a}(\lam)$. Then we have
\begin{enumerate}
\item [(1)] The minimizer $\gamma$ can be written as
\begin{align*}
\gamma=\sum\limits_{i=1}^{N'-1}|u_{i}\rangle \langle u_{i}|+(\lam-N'+1)|u_{N'}\rangle \langle u_{N'}|, \ \  \langle u_{i}, u_{j}\rangle=\delta_{ij},
\end{align*}
where $N'$ is the smallest integer such that $\lam\le N'$, and $(u_{1}, \cdots, u_{N'})$ satisfies
\begin{equation*}\label{uieqV}
H_{\gamma}u_{i}:=\big(\sqrt{-\Delta+m^{2}}-m-\frac{4}{3}a\rho_{\gamma}^{\frac{1}{3}}\big)u_{i}=\mu_{i}u_{i}\ \ \text{in} \ \ \R^{3}, \ \  i=1, \cdots, N'.
\end{equation*}
Here $\mu_{1}<\mu_{2}\le \cdots\le \mu_{N'}<0$ are the $N'$ first eigenvalues, counted with multiplicity, of the operator $H_{\gamma}$ on  $L^2(\R^{3},\C)$.

\item [(2)]  We have $(u_{1}, \cdots, u_{N'})\in \big(C(\R^{3}, \C)\big)^{N'}$. Moreover,
there exist positive constants $C_{-}>0$, $C_{0}>0$ and $C_{+}>0$, depending only on $m>0$ and $\mu_{1}<0$, such that 
\begin{equation*}\label{b332}
u_{1}(x)=\begin{cases}
\big(1+o(1)\big)C_{-}|x|^{-\frac{5}{2}}e^{-\theta_{1}|x|} &\text{as} \ \ |x|\to \infty, \ \ \text{if}\ \ m+\mu_{1}<0;\\
\big(1+o(1)\big)C_{0}|x|^{-\frac{3}{2}}e^{-\theta_{1}|x|} &\text{as} \ \ |x|\to \infty, \ \ \text{if}\ \ m+\mu_{1}=0;\\
\big(1+o(1)\big)C_{+}|x|^{-1}e^{-\theta_{1}|x|} &\text{as} \ \ |x|\to \infty, \ \ \text{if}\ \ m+\mu_{1}>0,
\end{cases}
\end{equation*}
where
\begin{equation*}\label{thetaidef}
0<\theta_{1}:=\begin{cases}
\sqrt{m^{2}-(m+\mu_{1})^{2}}, &\text{if} \ \ m+\mu_{1}>0;\\
m, &\text{if} \ \ m+\mu_{1}\le 0.
\end{cases}
\end{equation*}
\end{enumerate}
\end{lem}

Since the proof of Lemma \ref{prop34} is similar to that of  \cite[Proposition 2.2]{cb3},  for simplicity we omit the detailed proof.  We next introduce the following lemma, whose proof is left to Appendix A.

\begin{lem}\label{lem34}
For $N=1,2$, let $E_{a}(N)$  and $D_{4/3,N}$ be defined by  $(\ref{problem})$ and  $(\ref{LTineq})$, respectively. Then we have the following conclusions:
\begin{enumerate}
\item [(1)] If $0<a< D_{4/3,1}$, then $E_{a}(1)$ admits at least one minimizer. 	
\item [(2)]  If $0<a< D_{4/3,2}$, then $E_{a}(2)$ admits minimizers under the following condition
\begin{align}\label{3.13}
E_{a}(2)<2E_{a}(1).
\end{align}
\end{enumerate}
\end{lem}

Applying Lemmas \ref{lem32}--\ref{lem34},  we are now able to complete the proof of  Theorem \ref{th21}.

\vspace{0.2cm}

\noindent \textbf{Proof of Theorem \ref{th21}.} Using Lemma   \ref{prop34} (2),  the strict  binding inequality \eqref{3.13}  for $0<a< D_{4/3,2}$ can be proved by the similar  argument of  \cite[Theorem 1.1]{cb3}. Once the inequality \eqref{3.13} holds true, then the proof of Theorem \ref{th21} is complete in view of
 Lemmas \ref{lem32}, \ref{prop34} (1) and \ref{lem34} (2).   \qed

\section{Limiting Behavior of Minimizers as $a\nearrow D_{4/3,N}$}
In this section, we mainly prove Theorem \ref{th22} on the limiting behavior of minimizers for $E_{a}(2)$ as $a\nearrow D_{4/3,2}$, where $E_{a}(2)$ and $D_{4/3,2}>0$ are  defined by  \eqref{problem} and  (\ref{LTineq}), respectively.

Towards the above purpose, we define
\begin{equation}\label{dstar}
\begin{split}
d_{*}:=\inf \Big\{ & \tr\frac{\gamma}{\sqrt{-\Delta}}:\ \gamma \   \text{is an optimizer of}\  \\ \qquad\quad  & D_{4/3, 2} \ \text{satisfying}\  \|\gamma\|=\tr(\sqrt{-\Delta}\gamma)=1
\Big\}.
\end{split}
\end{equation}
The following lemma shows that the optimal constant $d_*\in(0, \infty)$ can be attained.

\begin{lem}\label{lem41}
The optimal constant $d_{*}$ satisfies $0<d_*< \infty$ and has an optimizer.
\end{lem}

\noindent \textbf{Proof.} By the Hardy-Kato inequality (cf. \cite[Chapter V.5.4]{kato})
\begin{align}\label{hk}
\frac{1}{|x|}\le \frac{\pi}{2}\sqrt{-\Delta}  \  \ \  \text{on}\    \   L^{2}(\R^{3}, \C),
\end{align}
one can calculate from Proposition \ref{th11} that for any $\gamma$ satisfying \eqref{dstar},
\begin{align*}
4=(\tr\gamma)^{2}\le \big(\tr\sqrt{-\Delta}\gamma\big)  \tr\frac{\gamma}{\sqrt{-\Delta}}=\tr\frac{\gamma}{\sqrt{-\Delta}}\le \frac{\pi}{2} \int_{\R^{3}}|x|\rho_{\gamma}dx<\infty,
\end{align*}
where the last inequality follows from the estimate \eqref{a15}.
This thus  yields that $d_{*}\in (0, \infty)$ holds true.

Let $\{\gamma_{n}\}=\{\sum_{j=1}^{2}|w^{n}_{j}\rangle \langle w^{n}_{j}|\}$ be a minimizing sequence of $d_{*}$,  where the orthonormal system $(w^{n}_{1}, w^{n}_{2})$ solves the following  nonlinear system
\begin{equation}\label{4.2}
\Big(\sqrt{-\Delta}-\frac{4}{3}D_{4/3,2}\rho_{\gamma_{n}}^{\frac{1}{3}}\Big)w^{n}_{j}=\mu^{n}_{j}w^{n}_{j}\ \  \text{in}\ \ \R^3
\end{equation}
for some $\mu_{1}^{n}<\mu_{2}^{n}<0$. Since $\tr\big(\sqrt{-\Delta}\gamma_{n}\big)=1$, we obtain that $\{w_j^{n}\}_n$ and  $\{\sqrt{\rho_{\gamma_{n}}}\}$ are bounded  uniformly in $H^{\frac{1}{2}}(\R^{3}, \C)$ for $j=1,2$. Note from (\ref{def:eN})  that
\begin{align}\label{lim}
\lim\limits_{n\to \infty} \int_{\R^{3}}\rho_{\gamma_{n}}^{\frac{4}{3}}dx=\frac{1}{D_{4/3,2}}>0.
\end{align}
This then implies  from \cite[Remark 2.10]{Concen} that the vanishing case of $\{\sqrt{\rho_{\gamma_{a_n}}}\}$  does not occur.
We thus deduce that there exist  $\{y_{n}\}\subset \R^{3}$ and  $(w_{1}, w_{2})\in \Big(H^{\frac{1}{2}}(\R^{3}, \C)\Big)^{2}\setminus \{0\}$ such that up to a subsequence if necessary,
\begin{align}\label{weak}
w_{j}^{n}(x+y_{n})\rightharpoonup w_{j}\ \ \text{weakly in}\ \  H^{\frac{1}{2}}(\R^{3}, \C)\ \ \text{as}\ \ n\to \infty, \ \, j=1,2.
\end{align}
We then derive from \eqref{4.2} that
\begin{align}\label{leq}
\Big(\sqrt{-\Delta}-\frac{4}{3}D_{4/3,2}\rho_{\gamma}^{\frac{1}{3}}\Big)w_{j}=\mu_{j}w_{j}\ \ \, \text{in}\, \ \R^3,\  \,  j=1,2,
\end{align}
where $\mu_{j}:=\lim\limits_{n\to\infty}\mu_{j}^{n}\in(-\infty, 0]$ and $\gamma:=\sum\limits_{j=1}^{2}|w_{j} \rangle \langle w_{j}|\neq0$. Multiplying $x\cdot\nabla w_j$ on both sides of (\ref{leq}) and  integrating over $\R^3$, we further obtain that
\begin{align}\label{ide}
\tr\big(\sqrt{-\Delta}\gamma\big)-\frac{3}{2}D_{4/3,2}\int_{\R^{3}}
\rho_{\gamma}^{\frac{4}{3}}dx=\frac{3}{2}\sum_{j=1}^{2}\mu_{j}\int_{\R^{3}}|w_{j}|^{2}dx.
 \end{align}

As a consequence of (\ref{leq}) and (\ref{ide}),  we obtain that
\begin{align*}
\tr\big(\sqrt{-\Delta}\gamma\big)=D_{4/3,2}\int_{\R^{3}}\rho_{\gamma}^{\frac{4}{3}}dx,
\end{align*}
and thus
\begin{align}\label{4.7a}
\tr\big(\sqrt{-\Delta}\gamma\big)\geq\|\gamma\|^{\frac{1}{3}}\tr\big(\sqrt{-\Delta}\gamma\big)\ge D_{4/3,2}\int_{\R^{3}}\rho_{\gamma}^{\frac{4}{3}}dx=\tr\big(\sqrt{-\Delta}\gamma\big),
\end{align}
which then  implies that $\|\gamma\|=1$ and $\gamma$ is an optimizer of $D_{4/3,2}$. Note from Theorem \ref{th11}  that $\|w_{j}\|_{2}^{2}=\|\gamma\|=1=\|w^{n}_{j}\|_{2}^{2}$ hold for $j=1, 2$.   We hence deduce from (\ref{weak}) that
\begin{align*}
w_{j}^{n}(x+y_{n})\to  w_{i}\ \ \text{strongly in}\ \  L^{2}(\R^{3}, \C)\ \ \text{as}\ \ n\to \infty, \ \ j=1,2.
\end{align*}
By the interpolation inequality, we conclude from (\ref{lim}) that
\begin{align*}
\lim\limits_{n\to \infty} \int_{\R^{3}}\rho_{\gamma_{n}}^{\frac{4}{3}}dx=\int_{\R^{3}}\rho_{\gamma}^{\frac{4}{3}}dx=\frac{1}{D_{4/3,2}},
\end{align*}
which further implies from \eqref{4.7a} that $\tr\big(\sqrt{-\Delta}\gamma\big)=1$. Therefore,  we have
$$
\tr \frac{\gamma}{\sqrt{-\Delta}}\geq d_{*}=\liminf\limits_{n\to\infty}\tr \frac{\gamma_n}{\sqrt{-\Delta}}\geq \tr \frac{\gamma}{\sqrt{-\Delta}},
$$
which yields that $\gamma$ is an optimizer of $d_{*}$. This proves Lemma \ref{lem41}. \qed

\subsection{Proof of Theorem \ref{th22}}
In this subsection, we complete the proof of Theorem \ref{th22}. We first establish the following $H^{\frac{1}{2}}$-convergence of  minimizers for $E_{a}(2)$ as $a\nearrow D_{4/3,2}$.


\begin{lem}\label{lem42}
Let $\gamma_{a}=\sum\limits_{i=1}^{2}|u^{a}_{i} \rangle \langle u^{a}_{i}|$ be a minimizer of $E_{a}(2)$, where $a\nearrow D_{4/3,2}$, and the orthonormal system $(u^{a}_{1},\,  u^{a}_{2})$ satisfies \eqref{1.18}.
Then there exist a subsequence $(u^{a_n}_{1},\,  u^{a_n}_{2})$ of $(u^{a}_{1},\,  u^{a}_{2})$ and a sequence $\{y_{a_n}\}\subset \R^{3}$ such that for $i=1,2$,
\begin{align}\label{d43}
w_{i}^{a_n}(x)&:=\eps_{a_n}^{\frac{3}{2}}u_{i}^{a_n}(\eps_{a_n}(x+y_{a_n}))\to w_{i}(x)\ \ \text{strongly in} \ \ H^{\frac{1}{2}}(\R^{3}, \C) \ \ \text{as}\ \ n\to\infty,
\end{align}
where $a_n\nearrow D_{4/3,2}$ as $n\to\infty$,  $0<\eps_{a_n}:=(\tr\sqrt{-\Delta}\gamma_{a_n})^{-1}\to 0$ as $n\to\infty$, and $\gamma:=\sum\limits_{i=1}^{2}|w_{i} \rangle \langle w_{i}|$ is an optimizer of $D_{4/3,2}$.
\end{lem}

\noindent \textbf{Proof.}  
We first claim  that
\begin{align}\label{4.4}
\sup\limits_{a\nearrow D_{4/3,2}}\tr(\sqrt{-\Delta}\gamma_{a})=\infty.
\end{align}
 On the contrary,  assume that
\begin{align}\label{d44}
\sup\limits_{a\nearrow D_{4/3,2}}\tr(\sqrt{-\Delta}\gamma_{a})<\infty.
\end{align}
Note from \eqref{def:eN} that  
\begin{equation}\label{4.7}
\begin{split}
\mathcal{E}_{D_{4/3,2}}(\gamma)
=&\lim\limits_{a\nearrow D_{4/3,2}}\Big[\mathcal{E}_{a}(\gamma)-(D_{4/3,2}-a)\inte\rho^{\frac{4}{3}}_{\gamma}dx\Big]\\
\ge&\lim\limits_{a\nearrow D_{4/3,2}}\Big[E_{a}(2)-\frac{D_{4/3,2}-a}{D_{4/3,2}}\text{Tr}(\sqrt{-\Delta}\gamma)\Big]\\
=&\lim\limits_{a\nearrow D_{4/3,2}}E_{a}(2),\ \, \forall\ \gamma\in\mathcal{P}_2,
\end{split}\end{equation}
where $\mathcal{P}_2$ is defined by \eqref{pn}. Optimizing \eqref{4.7} over  $\gamma\in\mathcal{P}_2$,  we hence obtain that
\begin{align}\label{4.8}
E_{D_{4/3,2}}(2)\geq\lim\limits_{a\nearrow D_{4/3,2}}E_{a}(2).
\end{align}
By the definition of $E_{a}(2)$,  we have
$$
E_{D_{4/3,2}}(2)\leq E_{a}(2),\ \, \forall\ 0<a\leq D_{4/3,2}.
$$
This then implies from \eqref{4.8} that
\begin{align*}
E_{D_{4/3,2}}(2)=\lim\limits_{a\nearrow D_{4/3,2}}E_{a}(2)=\lim\limits_{a\nearrow D_{4/3,2}}\mathcal{E}_{a}(\gamma_{a})\ge \lim\limits_{a\nearrow D_{4/3,2}}\mathcal{E}_{D_{4/3,2}}(\gamma_{a})\ge E_{D_{4/3,2}}(2),
\end{align*}
 and thus $\{\gamma_{a}\}$ is a minimizing sequence of $E_{D_{4/3,2}}(2)$. Consequently,  similar to the proof of Theorem \ref{th21} (1), we can  deduce from (\ref{d44}) that there exists $u_{i}\in H^{\frac{1}{2}}(\R^{3}, \C)$  such that up to translations and a subsequence if necessary,
\begin{align*}
u_{i}^{a}(x)\to u_{i}(x) \ \  \text{strongly in}\ \  H^{\frac{1}{2}}(\R^{3}, \C)\ \ \text{as}\ a\nearrow D_{4/3,2},\ \  i=1,2,
\end{align*}
which yields that $\sum_{i=1}^{2}|u_{i}\rangle \langle u_{i}|$ is a minimizer of $E_{D_{4/3,2}}(2)$. This however contradicts with Theorem \ref{th21} (2).  Therefore, the sequence $\big\{\tr(\sqrt{-\Delta}\gamma_{a})\big\}$ is unbounded as $a\nearrow D_{4/3,2}$.

We now obtain from \eqref{4.4} that there exists a subsequence $\{\gamma_{a_n}\}$ of $\{\gamma_a\}$ such that $$\lim\limits_{n\to\infty}\tr(\sqrt{-\Delta}\gamma_{a_n})=\infty,$$
where $\gamma_{a_n}:=\sum\limits_{i=1}^{2}|u^{a_n}_{i} \rangle \langle u^{a_n}_{i}|$, and  $a_n\nearrow D_{4/3,2}$ as $n\to\infty$.
Setting  $\eps_{a_n}:=(\tr\sqrt{-\Delta}\gamma_{a_n})^{-1}>0$, it then follows from above that  $\lim\limits_{n\to\infty}\eps_{a_n}=0$.
Denote
\begin{align*}\label{wiandef}
\bar{w}_{i}^{a_n}(x)&:=\eps_{a_n}^{\frac{3}{2}}u_{i}^{a_n}(\eps_{a_n}x), \  \  \ \bar{\gamma}_{a_n}:=\sum\limits_{i=1}^{2}|\bar{w}^{a_n}_{i} \rangle \langle \bar{w}^{a_n}_{i}|.
\end{align*}
Since $\tr\big(\sqrt{-\Delta}\bar{\gamma}_{a_n}\big)=1$,  it yields that $\{\sqrt{\rho_{\bar{\gamma}_{a_n}}}\}$ is bounded uniformly in $H^{\frac{1}{2}}(\R^3)$.
Note that
\begin{equation}\label{4.11}
\begin{split} 0&\ge \eps_{a_n}E_{a_n}(2)=\eps_{a_n}\mathcal{E}_{a_n}(\gamma_{a_n})\\
&=\eps_{a_n}\tr\big(\sqrt{-\Delta}\gamma_{a_n}\big)-a_n\eps_{a_n}\int_{\R^{3}}\rho_{\gamma_{a_n}}^{\frac{4}{3}}dx+o(1)\\
 &=\tr\big(\sqrt{-\Delta}\bar{\gamma}_{a_n}\big)-a_n\int_{\R^{3}}\rho_{\bar{\gamma}_{a_n}}^{\frac{4}{3}}dx\\
 &=1-D_{4/3,2}\int_{\R^{3}}\rho_{\bar{\gamma}_{a_n}}^{\frac{4}{3}}dx+o(1)\ \  \text{as}\  \ n\to\infty.
\end{split}
\end{equation}
This then implies from  \cite[Remark 2.10]{Concen}   that the vanishing case of $\{\sqrt{\rho_{\bar{\gamma}_{a_n}}}\}$  does not occur.
Hence,
there exist $\{y_{a_n}\}\subset \R^{3}$ and  $(w_1, w_2)\in \big(H^{\frac{1}{2}}(\R^{3}, \C)\big)^2\backslash\{0\}$ such that up to a subsequence if necessary,
\begin{align}\label{4.15}
w_{i}^{a_n}(x):=\bar{w}_{i}^{a_n}(x+y_{a_n})\rightharpoonup w_{i}\ \ \text{weakly in}\ \,  H^{\frac{1}{2}}(\R^{3}, \C)\ \ \text{as}\ \ n\to\infty, \ \ i=1,2.
\end{align}

Define $\gamma:=\sum\limits_{i=1}^{2}|w_{i} \rangle \langle w_{i}|$ and
\begin{equation*}\label{4.13}
\tilde{\gamma}_{a_n}:=\sum_{i=1}^{2}|w_{i}^{a_n}\rangle \langle w_{i}^{a_n}|=\sum_{i=1}^{2}\eps_{a_n}^{3}\big|u_{i}^{a_n}(\eps_{a_n}(\cdot+y_{a_n}))\big\rangle \big\langle u_{i}^{a_n}(\eps_{a_n}(\cdot+y_{a_n}))\big|.
\end{equation*}
We next prove that
\begin{align}\label{7}
\int_{\R^{3}}\rho_{\gamma}dx=\lim\limits_{n\to\infty}\int_{\R^{3}}\rho_{\tilde{\gamma}_{a_n}}dx.
\end{align}
On the contrary, suppose that
\begin{align}\label{8}
0<\int_{\R^{3}}\rho_{\gamma}dx<\lim\limits_{n\to\infty}\int_{\R^{3}}\rho_{\tilde{\gamma}_{a_n}}dx=2.
\end{align}
Using the strongly local convergence (cf. \cite[Lemma 7.3]{Duke}), there exists a sequence $\{R_{n}\}\subset \R$ satisfying $R_{n}\to \infty$ as $n\to\infty$ such that
\begin{align*}
0<\lim\limits_{n\to\infty}\int_{|x|\le R_{n}}\rho_{\tilde{\gamma}_{a_n}}dx=\int_{\R^{3}}\rho_\gamma dx<2, \ \  \lim\limits_{n\to\infty}\int_{R_{n}\le |x|\le 6R_{n} }\rho_{\tilde{\gamma}_{a_n}}dx=0.
\end{align*}
Take a cut-off function $\chi\in C_{0}^{\infty}(\R^{3}, [0,1])$ satisfying $\chi(x)=1$ for $|x|\le 1$ and $\chi(x)=0$ for $|x|\ge 2$. Define $\chi_{R_{n}}(x):=\chi(\frac{x}{R_{n}})$, $\eta_{R_{n}}(x):=\sqrt{1-\chi_{R_{n}}^{2}}$,
\[w_i^{1n}:=\chi_{R_{n}}w_i^{a_n}, \ \, w_i^{2n}:=\eta_{R_{n}}w_i^{a_n},\]
and
\[\gamma_{1n}:=\sum_{i=1}^2|w_i^{1n}\rangle\langle w_i^{1n}|,\ \, \gamma_{2n}:=\sum_{i=1}^2|w_i^{2n}\rangle\langle w_i^{2n}|.
\]
We then derive that
\begin{align}\label{loccv}
\rho_{\gamma_{1n}}\to \rho_\gamma\  \  \text{and}\ \ \eta_{R_{n}}^{2}\chi^{2}_{3R_{n}}\rho_{\tilde{\gamma}_{a_n}}\to 0 \  \ \text{strongly in} \   L^{1}(\R^{3})\  \text{as}\ n\to\infty.
\end{align}
Applying the IMS-type formula and  the interpolation inequality, we further deduce from (\ref{loccv}) that
\begin{equation}\label{4.19}
\begin{split}
\tr(\sqrt{-\Delta}\tilde{\gamma}_{a_n})&\ge \tr(\sqrt{-\Delta}\gamma_{1n})+\tr(\sqrt{-\Delta}\gamma_{2n})+o(1)\\
&\ge \tr(\sqrt{-\Delta}\gamma)+\tr(\sqrt{-\Delta}\gamma_{2n})+o(1)\ \, \text{as}\ \ n\to\infty,
\end{split}\end{equation}
and
\begin{align}\label{p2}
\lim\limits_{n\to\infty}\int_{\R^{3}}\rho_{\tilde{\gamma}_{a_n}}^{\frac{4}{3}}dx
&=\lim\limits_{n\to\infty}\int_{\R^{3}}\Big(\chi^{2}_{R_{n}}\rho_{\tilde{\gamma}_{a_n}}+\eta_{R_{n}}^{2}\chi^{2}_{3R_{n}}\rho_{\tilde{\gamma}_{a_n}}+\eta_{3R_{n}}^{2}\rho_{\tilde{\gamma}_{a_n}}
\Big)^{\frac{4}{3}}dx\nonumber\\
&=\lim\limits_{n\to\infty}\int_{\R^{3}}\Big(\chi^{2}_{R_{n}}\rho_{\tilde{\gamma}_{a_n}}+\eta_{3R_{n}}^{2}\rho_{\tilde{\gamma}_{a_n}}
\Big)^{\frac{4}{3}}dx\\
&=\lim\limits_{n\to\infty}\int_{\R^{3}}(\rho_{\gamma_{1n}}^{\frac{4}{3}}+\rho_{\gamma_{2n}}^{\frac{4}{3}})dx=\int_{\R^{3}}\rho_{\gamma}^{\frac{4}{3}}dx+\lim\limits_{n\to\infty}\int_{\R^{3}}\rho_{\gamma_{2n}}^{\frac{4}{3}}dx.\nonumber
\end{align}

 Note from \eqref{def:eN} and \eqref{4.11} that
\begin{align}\label{d48}
0&\ge\lim\limits_{n\to\infty} \eps_{a_n}E_{a_n}(2)
=\lim\limits_{n\to\infty} \Big[\tr\big(\sqrt{-\Delta}\tilde{\gamma}_{a_n}\big)-D_{4/3,2}\int_{\R^{3}}\rho_{\tilde{\gamma}_{a_n}}^{\frac{4}{3}}dx\Big]\geq0.
\end{align}
We then calculate from (\ref{def:eN}) and (\ref{4.19})--(\ref{d48}) that
\begin{align}\label{c4381}
0&=\lim\limits_{n\to\infty}\eps_{a_n}E_{a_n}(2)
=\lim\limits_{n\to\infty}\Big[\tr(\sqrt{-\Delta}\tilde{\gamma}_{a_n})-D_{4/3,2}\int_{\R^{3}}
\rho_{\tilde{\gamma}_{a_n}}^{\frac{4}{3}}dx\Big]\nonumber\\
&\ge \tr\big(\sqrt{-\Delta}\gamma\big)-D_{4/3,2}\int_{\R^{3}}\rho_{\gamma}^{\frac{4}{3}}dx\nonumber\\
&\quad+\lim\limits_{n\to\infty}\Big\{\tr(\sqrt{-\Delta}\eta_{R_n}\tilde{\gamma}_{a_n}\eta_{R_n})-D_{4/3,2}\int_{\R^{3}}(\eta_{R_n}^{2}\rho_{\tilde{\gamma}_{a_n}})^{\frac{4}{3}}dx\Big\}\nonumber\\
&\ge \|\gamma\|^{\frac{1}{3}}\tr\big(\sqrt{-\Delta}\gamma\big)-D_{4/3,2}\int_{\R^{3}}\rho_{\gamma}^{\frac{4}{3}}dx\\
&\quad+\lim\limits_{n\to\infty}\Big\{\|\eta_{R_n}\tilde{\gamma}_{a_n}\eta_{R_n}\|^{\frac{1}{3}}\tr(\sqrt{-\Delta}\eta_{R_n}\tilde{\gamma}_{a_n}\eta_{R_n})-D_{4/3,2}\int_{\R^{3}}(\eta_{R_n}^{2}\rho_{\tilde{\gamma}_{a_n}})^{\frac{4}{3}}dx\Big\}\nonumber\\
&\ge \|\gamma\|^{\frac{1}{3}}\tr\big(\sqrt{-\Delta}\gamma\big)-D_{4/3,2}\int_{\R^{3}}\rho_{\gamma}^{\frac{4}{3}}dx\ge 0,\nonumber
\end{align}
where we have used the facts that  $\|\gamma\|\le \liminf\limits_{n\to\infty}\|\tilde{\gamma}_{a_n}\|=1$ and $\|\eta_{R_n}\tilde{\gamma}_{a_n}\eta_{R_n}\|\le \|\tilde{\gamma}_{a_n}\|=1$. This thus  implies that
$\gamma$ is an optimizer of $D_{4/3,2}$ with $\|\gamma\|=1$.
Together with  Theorem \ref{th11},  we  conclude that
\begin{align*}
\int_{\R^{3}}\rho_{\gamma}dx=2=\lim\limits_{n\to\infty}\int_{\R^{3}}\rho_{\tilde{\gamma}_{a_n}}dx,
\end{align*}
which however contradicts with \eqref{8}. This implies that  \eqref{7} holds true.

By \eqref{4.15} and \eqref{7}, we have
\begin{align*}
w_i^{a_n}\to w_i\ \, \text{strongly in}\ \ L^r(\R^3,\C)\ \ \text{as}\ \ n\to\infty, \ \ \forall\ r\in[2, 3),
\end{align*}
which yields from \eqref{c4381}  that
\begin{equation*}\label{4.17}
\lim\limits_{n\to\infty}\tr\big(\sqrt{-\Delta}\tilde{\gamma}_{a_n}\big)=\tr\big(\sqrt{-\Delta}\gamma\big),
\end{equation*}
and hence \eqref{d43} holds true.  This therefore completes the proof of  Lemma \ref{lem42}.  \qed

\vspace{0.05cm}
We next  study the uniform decaying estimate of $\{w_{i}^{a_n}\}$ as $n\to\infty$,  which  will be used to analyze the precise blow-up rate of
minimizers for $E_{a_n}(2)$ as $n\to\infty$.

\begin{lem}\label{lem44}
Let $\{w_{i}^{a_n}\}_n$ be given by  $(\ref{d43})$ for $i=1, 2$, where $a_n\nearrow D_{4/3,2}$ as $n\to\infty$. Then there exists a constant $C>0$, independent of $n>0$, such that for sufficiently large $n>0$,
\begin{align}\label{wndecay}
|w_{i}^{a_n}(x)|\le C(1+|x|^{4})^{-1}\ \, \text{in}\ \,  \R^{3},\ \, i=1,2.
\end{align}
\end{lem}
\noindent \textbf{Proof.}
We first prove that
\begin{align}\label{d426}
\{w_{i}^{a_n}\}_n \ \  \text{is bounded uniformly in}\     H^{1}(\R^{3}, \C),\ \, i=1,2.
\end{align}
Recall from  \eqref{d43} that $w_{i}^{a_n}(x)\equiv\eps_{a_n}^{\frac{3}{2}}u_{i}^{a_n}(\eps_{a_n}(x+y_{a_n}))$ in $\R^3$, where  $\sum\limits_{i=1}^{2}|u^{a_n}_{i} \rangle \langle u^{a_n}_{i}|$ is a minimizer of $E_{a_n}(2)$ satisfying \eqref{1.18}.  It then follows that
\begin{align}\label{sys2}
\Big[\sqrt{-\Delta+\eps_{a_n}^{2}m^{2}}-\eps_{a_n}m-\frac{4a_n}{3}\rho_{\tilde{\gamma}_{a_n}}^{\frac{1}{3}}\Big]
w^{a_n}_{i}=\eps_{a_n}\mu_{i}^{a_n}w_{i}^{a_n}\ \ \text{in}\ \,  \R^{3},\ \, i=1,2,
\end{align}
where $\mu_{1}^{a_n}<\mu_{2}^{a_n}<0$, $\tilde{\gamma}_{a_n}=\sum_{i=1}^2|w_{i}^{a_n}\rangle\langle w_{i}^{a_n}|$, and $\eps_{a_n}\to 0^+$ as $n\to\infty$.
This yields that
\begin{align}\label{c447}
\sum\limits_{i=1}^{2}\mu_{i}^{a_n}\eps_{a_n}
&=\tr\big(\sqrt{-\Delta+\eps_{a_n}^{2}m^{2}}-\eps_{a_n}m\big)\tilde{\gamma}_{a_n}-
\frac{4a_n}{3}\int_{\R^{3}}\rho_{\tilde{\gamma}_{a_n}}^{\frac{4}{3}}dx,
\end{align}
and
\begin{align}\label{4.32}
&\quad\tr(-\Delta+\eps_{a_n}^{2}m^{2})\tilde{\gamma}_{a_n}\nonumber\\
&=\sum_{i=1}^{2}\big\langle \sqrt{-\Delta+\eps_{a_n}^{2}m^{2}}w_{i}^{a_n}, \sqrt{-\Delta+\eps_{a_n}^{2}m^{2}}w_{i}^{a_n} \big\rangle \nonumber\\
&=\sum_{i=1}^{2}\big\langle \eps_{a_n}mw_{i}^{a_n}+\eps_{a_n}\mu_{i}^{a_n}w_{i}^{a_n}+\frac{4a_n}{3}\rho_{\tilde{\gamma}_{a_n}}^{\frac{1}{3}}w^{a_n}_{i}, \eps_{a}mw_{i}^{a_n}+\eps_{a_n}\mu_{i}^{a_n}w_{i}^{a_n}+\frac{4a_n}{3}\rho_{\tilde{\gamma}_{a_n}}^{\frac{1}{3}}w^{a_n}_{i}\big\rangle\nonumber\\
&=2\eps_{a_n}^{2}m^{2}+\eps_{a_n}^{2}\sum_{i=1}^{2}|\mu_{i}^{a_n}|^{2}+\frac{16a_n^{2}}{9}\int_{\R^{3}}\rho_{\tilde{\gamma}_{a_n}}^{\frac{5}{3}}dx+2\eps_{a_n}^{2}m\sum_{i=1}^{2}\mu_{i}^{a_n}\\
&\quad+\frac{8a_n}{3}\eps_{a_n}m\int_{\R^{3}}\rho_{\tilde{\gamma}_{a_n}}^{\frac{4}{3}}dx+\frac{8a_n}{3}\sum_{i=1}^{2}\eps_{a_n}\mu_{i}^{a_n}\int_{\R^{3}}\rho_{\tilde{\gamma}_{a_n}}^{\frac{1}{3}}|w^{a_n}_{i}|^{2}dx\nonumber\\
&:=I_{a_n}+\frac{16a_n^{2}}{9}\int_{\R^{3}}\rho_{\tilde{\gamma}_{a_n}}^{\frac{5}{3}}dx\nonumber.
\end{align}

Using \eqref{d43} and \eqref{c447},  we obtain that up to a subsequence if necessary,
\begin{equation*}
\begin{split}
\lim\limits_{n\to\infty}\sum\limits_{i=1}^{2}\mu_{i}^{a_n}\eps_{a_n}=\tr\big(\sqrt{-\Delta}\gamma\big)-\frac{4D_{4/3,2}}{3}\int_{\R^{3}}\rho_{\gamma}^{\frac{4}{3}}dx
=-\frac{1}{3}\int_{\R^{3}}\rho_{\gamma}^{\frac{4}{3}}dx,
\end{split}\end{equation*}
where $\gamma$ given by Lemma \ref{lem42} is an optimizer of  $D_{4/3,2}$.  This  implies that
\begin{equation}\label{mu}
\mu_i:=\lim\limits_{n\to\infty}\mu_{i}^{a_n}\eps_{a_n}\in(-\infty, 0],\ \, i=1,2.
\end{equation}
Together with the uniform boundedness of $\{w_i^{a_n}\}_n$ in  $H^{\frac{1}{2}}(\R^3, \C)$ for $i=1,2$, we  thus conclude  that the term $I_{a_n}$ of \eqref{4.32} satisfies
\begin{align}\label{Inbd}
\sup\limits_{n}I_{a_n}<\infty.
\end{align}
By Young's inequality and the interpolation inequality, one can calculate that  for  sufficiently large $n>0$,
\begin{align}\label{d431}
\int_{\R^{3}}\rho_{\tilde{\gamma}_{a_n}}^{\frac{5}{3}}dx&\le \epsilon^{-2}\int_{\R^{3}}\rho_{\tilde{\gamma}_{a_n}}dx+ \epsilon\int_{\R^{3}}\rho_{\tilde{\gamma}_{a_n}}^{2}dx\nonumber\\
&\le 2\epsilon^{-2}+2\epsilon\sum_{i=1}^{2}\int_{\R^{3}}|w_{i}^{a_n}|^{4}dx\nonumber\\
&\le 2\epsilon^{-2}+2\epsilon\sum_{i=1}^{2}\Big(\int_{\R^{3}}|w^{a_n}_{i}|^{3}dx\Big)^{\frac{2}{3}}\Big(\int_{\R^{3}}|w^{a_n}_{i}|^{6}dx\Big)^{\frac{1}{3}}\\
&\le 2\epsilon^{-2}+2\epsilon S_{6}^{-1}\sum_{i=1}^{2}\Big(\int_{\R^{3}}|w^{a_n}_{i}|^{3}dx\Big)^{\frac{2}{3}}\int_{\R^{3}}(|\nabla w_{1}^{a_n}|^{2}+|w_{i}^{a_n}|^{2})dx\nonumber\\
&\le 2\epsilon^{-2}+2C+\epsilon C\tr(-\Delta \tilde{\gamma}_{a_n}),\ \, \forall\  \epsilon>0,\nonumber
\end{align}
where $S_{6}^{-1}=\inf\limits_{u\in H^{1}(\R^{3})}\dfrac{\|u\|^{2}_{H^{1}}}{\|u\|^{2}_{6}}$, and the constant $C>0$ is independent of $n>0$ and $\epsilon>0$.
Choose $\epsilon>0$ small enough so that $\frac{\epsilon16CD^2_{4/3,2}}{9}<1$.
Applying  (\ref{Inbd}) and (\ref{d431}),  we therefore deduce from  (\ref{4.32})  that $\{\tr(-\Delta \tilde{\gamma}_{a_n})\}$ is bounded uniformly in $n>0$. This proves  (\ref{d426}).

By the $H^{\frac{1}{2}}$-convergence of \eqref{d43} and the interpolation inequality, we deduce from  (\ref{d426}) that
\begin{equation*}\label{winlr1}
w^{a_n}_{i}\to w_{i} \ \ \text{strongly in}\ \  L^{r}(\R^{3}, \C)\ \ \text{as}\ \ n\to\infty,\ \,  \forall\ r\in[2, 6),
\end{equation*}
and thus
\begin{align}\label{winlr}
\rho_{\tilde{\gamma}_{a_n}}^{\frac{1}{3}}w^{a_n}_{i}\to\rho_{\gamma}^{\frac{1}{3}}w_{i} \  \ \text{strongly in}\   \  L^{r}(\R^{3}, \C)\ \ \text{as}\ \ n\to\infty,\  \  \forall\ r\in[6/5, 18/5),
\end{align}
where  $\tilde{\gamma}_{a_n}=\sum_{i=1}^2|w_{i}^{a_n}\rangle\langle w_{i}^{a_n}|$ and $\gamma=\sum_{i=1}^2|w_i\rangle\langle w_i|$.
By \eqref{sys2} and  \eqref{mu}, we have
\begin{align}\label{4.31}
\big[\sqrt{-\Delta}-\frac{4}{3}D_{4/3,2}\rho_\gamma^{\frac{1}{3}}\big]w_i=\mu_iw_i\ \,\ \text{in}\ \, \R^3, \ \ i=1, 2.
\end{align}
Let $G_{i}^{a_n}$ be  the Green's function of the operator \[\sqrt{-\Delta+\eps_{a_n}^{2}m^{2}}-\eps_{a_n}m-\eps_{a_n}\mu_{i}^{a_n}\ \ \mbox{in}\,\ \R^3,\, \ i=1, 2.\]
Since $\gamma=\sum_{i=1}^2|w_i\rangle\langle w_i|$ is an optimizer of $D_{4/3,2}$,  it yields from \eqref{4.31} and Theorem \ref{th11} that
\begin{align}\label{4.33}
\lim\limits_{n\to\infty}\eps_{a_n}\mu_{i}^{a_n}=\mu_i<0,\  \  i=1, 2.
\end{align}
The same argument of \cite[Lemma 2.4]{cb3} thus gives that there exists a constant $C_{g}>0$, independent of $i,m, n$ and $x$,  such that for  sufficiently large $n>0$,
\begin{align}\label{gi1}
0<G_{i}^{a_n}(x)<C_{g}|x|^{-2}\  \   \,\text{in}\   \,  \R^{3},\ \, i=1,2.
\end{align}

Denote $f_i^{a_n}:=\frac{4a_{n}}{3}\rho_{\tilde{\gamma}_{a_n}}^{\frac{1}{3}}w^{a_n}_{i}$ for $i=1, 2$. By (\ref{winlr}) and  (\ref{gi1}), one can calculate from (\ref{sys2})  that
for $i=1,2$,
\begin{align}\label{d437}
&\quad\ \|w_{i}^{a_n}\|_{\infty}\nonumber\\
&=\|G_{i}^{a_n}\ast f_i^{a_n}\|_\infty\nonumber\\
&\le \sup_{x\in \R^{3}}\int_{\R^{3}}G_{i}^{a_n}(x-y)|f^{a_n}_{i}(y)|dy\nonumber\\
&\le C_g\sup_{x\in \R^{3}}\Big[\int_{|x-y|\le 1}|x-y|^{-2}|f^{a_n}_{i}(y)|dy
+\int_{|x-y|>1}|x-y|^{-2}|f^{a_n}_{i}(y)|dy\Big]\\
&\le C_g\|f^{a_n}_{i}\|_{16/5}\Big(\int_{|y|\le 1}|y|^{-\frac{32}{11}}dy\Big)^{\frac{11}{16}}+C_g\|f_{i}^{a_n}\|_{5/2}\Big(\int_{|y|>1}|y|^{-\frac{10}{3}}dy\Big)^{\frac{3}{5}}\nonumber\\
&\le C(\|f^{a_n}_{i}\|_{16/5}+\|f_{i}^{a_n}\|_{5/2})<\infty \nonumber\ \  \text{uniformly for sufficiently large}\ n>0,
\end{align}
and
\begin{align}\label{d438}
\lim\limits_{|x|\to \infty}|w_{i}^{a_n}(x)|
&\le \lim\limits_{R\to \infty}\lim\limits_{|x|\to \infty}\Big(\int_{|y|\le R}G_{i}^{a_n}(x-y)|f^{a_n}_{i}(y)|dy\nonumber\\
&\qquad\qquad\qquad\quad +\int_{|y|\ge R}G_{i}^{a_n}(x-y)|f^{a_n}_{i}(y)|dy\Big)\nonumber\\
&=\lim\limits_{R\to \infty}\lim\limits_{|x|\to \infty}\int_{|y|\ge R}G_{i}^{a_n}(x-y)|f^{a_n}_{i}(y)|dy\nonumber\\
&\le \lim\limits_{R\to \infty}\lim\limits_{|x|\to \infty}  \int_{\{y: \ |y|\ge R, \ |x-y|\le 1\}}G_{i}^{a_n}(x-y)|f^{a_n}_{i}(y)|dy \\
&\quad+ \lim\limits_{R\to \infty}\lim\limits_{|x|\to \infty}  \int_{\{y: \ |y|\ge R, \ |x-y|>1\}}G_{i}^{a_n}(x-y)|f^{a_n}_{i}(y)|dy   \nonumber\\
&\le C\lim\limits_{R\to \infty}\Big(\|f_{i}^{a_n}\|_{L^{16/5}(B_{R}^{c})}+\|f_{i}^{a_n}\|_{L^{5/2}(B_{R}^{c})}\Big) \nonumber\\
&=0 \ \  \text{uniformly for sufficiently large}\ n>0.\nonumber
\end{align}
Therefore, similar to the proof of (\ref{a15}), we deduce from (\ref{sys2}), (\ref{4.33}), (\ref{d437}) and (\ref{d438}) that the estimate (\ref{wndecay}) holds true. This completes the proof of Lemma \ref{lem44}.\qed

We are now ready to finish the proof of Theorem \ref{th22}.
\vspace{0.05cm}

\noindent \textbf{Proof of Theorem \ref{th22}.}  Let  $\{w_{i}^{a_n}\}$  and $w_{i}$ be given by Lemma \ref{lem42}, $i=1, 2$, where $a_n\nearrow D_{4/3,2}$ as $n\to\infty$. We first prove that  \eqref{epsd} holds true.

Note  from (\ref{def:eN}) that
\begin{align}\label{es1}
&\quad E_{a_n}(2)+2m=\mathcal{E}_{a_n}(\gamma_{a_n})+2m\nonumber\\
&=\tr\big(\sqrt{-\Delta}\gamma_{a_n}\big)-a_n\int_{\R^{3}}\rho_{\gamma_{a_n}}^{\frac{4}{3}}dx+\tr\big(\sqrt{-\Delta+m^{2}}-\sqrt{-\Delta}\, \big)\gamma_{a_n}\nonumber\\
&\ge \Big(1-\dfrac{a_n}{D_{4/3,2}}\Big)\tr\big(\sqrt{-\Delta}\gamma_{a_n}\big)+\tr\big(\sqrt{-\Delta+m^{2}}-\sqrt{-\Delta}\, \big)\gamma_{a_n}\\
&=\eps^{-1}_{a_n}\Big(1-\dfrac{a_n}{D_{4/3,2}}\Big)\tr\big(\sqrt{-\Delta}\tilde{\gamma}_{a_n}\big)+\eps^{-1}_{a_n}\tr\big(\sqrt{-\Delta+\eps^2_{a_n}m^{2}}-\sqrt{-\Delta}\, \big)\tilde{\gamma}_{a_n}\nonumber\\
&=\eps^{-1}_{a_n}\Big(1-\dfrac{a_n}{D_{4/3,2}}\Big)+\eps_{a_n}m^{2}\tr\frac{\tilde{\gamma}_{a_n}}{\sqrt{-\Delta+\eps^2_{a_n}m^{2}}+\sqrt{-\Delta}},\nonumber
\end{align}
where $\gamma_{a_n}$ and $\eps_{a_n}$ are  as in Lemma \ref{lem42}, and $\tilde{\gamma}_{a_n}=\sum_{i=1}^2|w_{i}^{a_n}\rangle\langle w_{i}^{a_n}|$.
Applying   Hardy-Kato inequality \eqref{hk}  and  Dominated Convergence Theorem, one can check  from \eqref{a15}, \eqref{d43} and  \eqref{wndecay}  that
\begin{align*}\label{es2}
\lim\limits_{n\to\infty}\tr\frac{\tilde{\gamma}_{a_n}}{\sqrt{-\Delta+\eps^2_{a_n}m^{2}}+\sqrt{-\Delta}}=\tr\frac{\gamma}{2\sqrt{-\Delta}}.
\end{align*}
This thus  implies from  (\ref{es1})  that
\begin{equation}\label{left}
\begin{split}
E_{a_n}(2)+2m&\ge  \eps^{-1}_{a_n}\dfrac{D_{4/3,2}-a_n}{D_{4/3,2}}
+\eps_{a_n}\big(1+o(1)\big)\tr\frac{m^2\gamma}{2\sqrt{-\Delta}}\\
&\ge \sqrt{2}m\big(1+o(1)\big) \Big(\dfrac{D_{4/3,2}-a_n}{D_{4/3,2}}\tr\frac{\gamma}{\sqrt{-\Delta}}\Big)^{\frac{1}{2}} \\
&\ge \sqrt{2}m\big(1+o(1)\big) \Big(\dfrac{D_{4/3,2}-a_n}{D_{4/3,2}}d_{*}\Big)^{\frac{1}{2}} \ \ \text{as}\ \ n\to\infty,
\end{split}
\end{equation}
where the identity holds if and only if $\gamma$ is an optimizer of $d_{*}$, and
\begin{align}\label{appr}
\eps_{a_n}\approx  \Big(\dfrac{2(D_{4/3,2}-a_n)}{m^2d_{*}D_{4/3,2}}\Big)^{\frac{1}{2}}
\ \ \text{as}\ \ n\to\infty.
\end{align}
Here $d_*\in(0, \infty)$ is given by \eqref{dstar}.

On the other hand, denote $\gamma^{*}_{t_n}:=t_{n}^{3}\gamma^{*}(t_nx, t_ny)$,  where $\gamma^{*}$ is an optimizer of $d_{*}$, and
\begin{align*}
t_n=\Big(\dfrac{m^2D_{4/3,2}}{2(D_{4/3,2}-a_n)}\tr\frac{\gamma^{*}}{\sqrt{-\Delta}}\Big)^{\frac{1}{2}},\ \,  \forall\, a_n\in (0, D_{4/3,2}).
\end{align*}
By the definition of $d_{*}$, it yields that $\tr\big(\sqrt{-\Delta}\gamma^*\big)=1$, and $\gamma^*$ is an optimizer of $D_{4/3,2}$. We hence calculate  that
\begin{align}\label{es3}
E_{a_n}(2)+2m&\le \mathcal{E}_{a_n}(\gamma^{*}_{t_n})+2m
\nonumber\\
&=\tr\big(\sqrt{-\Delta+m^{2}}-\sqrt{-\Delta}\, \big)\gamma^{*}_{t_n}+\tr\big(\sqrt{-\Delta}\gamma^{*}_{t_n}\big)-a_n\int_{\R^{3}}\rho_{\gamma^{*}_{t_n}}^{\frac{4}{3}}dx\nonumber\\
&=t_n\big(1-a_n/D_{4/3,2}\big)+t_n\tr\big(\sqrt{-\Delta+t_n^{-2}m^{2}}-\sqrt{-\Delta}\, \big)\gamma^{*}\nonumber\\
&=t_n\big(1-a_n/D_{4/3,2}\big)+t_n^{-1}m^2\tr\frac{\gamma^{*}}{\sqrt{-\Delta+t_n^{-2}m^{2}}+\sqrt{-\Delta}}\\
&\leq t_n\big(1-a_n/D_{4/3,2}\big)+t_n^{-1}m^2\tr\frac{\gamma^{*}}{2\sqrt{-\Delta}}\nonumber\\
&=\sqrt{2}m \Big(\dfrac{D_{4/3,2}-a_n}{D_{4/3,2}}d_{*}\Big)^{\frac{1}{2}} ,\ \, \forall\ a_n\in (0, D_{4/3,2}).\nonumber
\end{align}
As a consequence of  (\ref{left})--(\ref{es3}), we obtain that  (\ref{epsd}) holds true.

To establish Theorem \ref{th22},  the rest is to prove that
\begin{align}\label{cvinfy}
w_{i}^{a_n}(x)\to w_{i}(x) \ \ \text{strongly in}  \ L^{\infty}(\R^{3}, \C) \ \ \text{as}\ \ n\to\infty,\ \, i=1,2.
\end{align}
Actually, recall from (\ref{sys2}) that
\begin{align}\label{c457}
\hat{w}^{a_n}_{i}(\xi)=\frac{\frac{4}{3}a_n(\rho^{\frac{1}{3}}_{\tilde{\gamma}_{a_n}}w^{a_n}_{i})\string^(\xi)}{\sqrt{|\xi|^{2}+\eps_{a_n}^{2}m^{2}}-\eps_{a_n}(m+\mu_{i}^{a_n})}\ \ \text{in}\ \, \R^3,\ \, i=1,2.
\end{align}
We then calculate  from \eqref{d426} and \eqref{d437} that 
\begin{equation*}\label{c459}
\begin{split}
\int_{\R^{3}}|\xi|^{4}|\hat{w}^{a_n}_{i}(\xi)|^{2}d\xi
&=\frac{16}{9}a_n^{2}\int_{\R^{3}}\frac{|\xi|^{4}\, \big|(\rho^{\frac{1}{3}}_{\tilde{\gamma}_{a_n}}w^{a_n}_{i})\string^(\xi)\big|^{2}}{\big(\sqrt{|\xi|^{2}+\eps_{a_n}^{2}m^{2}}-\eps_{a_n}(m+\mu_{i}^{a_n})\big)^{2}}\nonumber\\
&\le C\int_{\R^{3}}|\xi|^{2}|(\rho^{\frac{1}{3}}_{\tilde{\gamma}_{a_n}}w^{a_n}_{i})\string^(\xi)|^{2}d\xi\\
&=C\int_{\R^{3}}|\nabla (\rho^{\frac{1}{3}}_{\tilde{\gamma}_{a_n}}w^{a_n}_{i})|^{2}dx\\
&\le \tilde{C}\int_{\R^{3}} \rho^{\frac{2}{3}}_{\tilde{\gamma}_{a_n}}\sum_{i=1}^{2}|\nabla w^{a_n}_{i}|^{2}dx\\
&\le \tilde{C}\sup_{n\ge 1} \|\rho^{\frac{2}{3}}_{\tilde{\gamma}_{a_n}}\|_{\infty}\int_{\R^{3}}\sum_{i=1}^{2}|\nabla w^{a_n}_{i}|^{2}dx\nonumber\\
&<\infty \ \ \text{uniformly for sufficiently large}\ n>0,\nonumber
\end{split}\end{equation*}
which indicates that $\{w^{a_n}_{i}\}_n$ is bounded uniformly in $W^{2,2}(\R^{3}, \C)$ for $i=1,2$. Since the embedding $W^{2,2}(B_{R}, \C)\hookrightarrow C(B_{R}, \C)$ is compact for any $R>0$ (see \cite[Theorem 7.26]{GT}),   we derive that
\begin{align}\label{c456}
w^{a_n}_{i}\to w_{i}\ \ \text{uniformly on any compact domain of $\R^{3}$}\ \  \text{as} \ \ n\to\infty, \  \,  i=1, 2.
\end{align}
Moreover, it follows  from  (\ref{a15}) and  (\ref{wndecay}) that for any $\eps>0$, there exists a sufficiently large constant $R_{\eps}>0$, independent of $n>0$, such that for  sufficiently  large $n>0$,
\begin{align}\label{c455}
\sup_{|x|\ge R_{\eps}} |w^{a_n}_{i}(x)-w_{i}(x)|\le \sup_{|x|\ge R_{\eps}} \big(|w^{a_n}_{i}(x)|+|w_{i}(x)|\big)<\frac{\eps}{2}, \  \ i=1, 2.
\end{align}
We thus derive from (\ref{c456}) and (\ref{c455}) that (\ref{cvinfy}) holds true. This therefore  completes the proof of Theorem \ref{th22}. \qed

We finally remark from (\ref{left}) and (\ref{es3}) that the energy estimate (\ref{spi}) holds true.

\appendix
\section{Appendix}

The purpose of this appendix is to address the proof of Lemma \ref{lem34}.   Throughout this appendix, we assume that  $a\in(0, D_{4/3,1})$, where the best constant $D_{4/3,1}>0$ is defined by \eqref{def:eN}.

\vspace{0.20cm}

\noindent \textbf{Proof of Lemma \ref{lem34}.}  (1). Let  $\{|u_{n}\rangle \langle u_{n}|\}$ be a minimizing sequence of $E_{a}(1)$. Since
\begin{align*}
\mathcal{E}_{a}(|u_{n}\rangle \langle u_{n}|)
&\ge \langle u_{n}, \sqrt{-\Delta}u_{n}\rangle-a\int_{\R^{3}}|u_{n}|^{\frac{8}{3}}dx-m\nonumber\\
&\ge \Big(1-\frac{a}{D_{4/3,1}}\Big)\langle u_{n}, \sqrt{-\Delta}u_{n}\rangle-m,
\end{align*}
one gets  that $\{u_{n}\}$ is bounded uniformly in $H^{\frac{1}{2}}(\R^{3}, \C)$.

Since $E_{a}(1)<0$ holds in view of
Lemma \ref{lem32} (1), it can be checked from \cite[Remark 2.10]{Concen} that the vanishing case of $\{u_n\}$  does not occur, $i.e.$,  there exists $R>0$ such that
\begin{align}\label{b330}
\lim\limits_{n\to \infty}\sup_{y\in \R^{3}}\int_{B_{R}(y)}|u_{n}|^{2}dx>0.
\end{align}
This then yields that  there exists a sequence $\{y_{n}\}\subset\R^3$ such that 
\begin{align}\label{b331}
\lim\limits_{n\to \infty}\int_{B_{R}(y_{n})}|u_{n}|^{2}dx>0,
\end{align}
where $R>0$ is as in \eqref{b330}.
Since the energy functional $\mathcal{E}_a(\cdot)$ is translationally invariant, without loss of generality, we can suppose  $y_n=0$. We thus deduce from \eqref{b331} that  there exists $u\in H^{\frac{1}{2}}(\R^{3},\C)\backslash\{0\}$ such that up to a subsequence if necessary,
\begin{align}\label{a.2}
u_{n}\rightharpoonup u\  \  \text{weakly in} \  \ H^{\frac{1}{2}}(\R^{3},\C) \  \text{as}\ n\to \infty.
\end{align}

By Br\'ezis-Lieb Lemma (cf. \cite{minimax}) and the interpolation inequality,  one can verify from \eqref{a.2} that  if $\|u\|^{2}_{2}=1$,  then
\begin{align*}
u_{n}\to u \  \text{strongly in} \  \ L^{r}(\R^{3},\C) \  \text{as}\ n\to \infty, \ \ \forall\ r\in [2,3),
\end{align*}
and thus
\begin{align*}
E_{a}(1)=\lim\limits_{n\to \infty}\mathcal{E}_{a}(|u_{n}\rangle \langle u_{n}|) \ge \mathcal{E}_{a}(|u\rangle \langle u|) \ge E_{a}(1),
\end{align*}
which further yields that $|u\rangle \langle u|$ is a minimizer of $E_{a}(1)$.  Therefore, in order to complete the proof of  Lemma \ref{lem34} (1),  we only need to prove that  $\|u\|^{2}_{2}=1$.

On the contrary,  assume that $0<\lambda:=\|u\|_2^{2}< 1$.    By the strongly local convergence (cf. [21, Lemma 7.3]), there exists a sequence $\{R_{n}\}$ satisfying $R_{n}\to \infty$ as $n\to \infty$ such that
\begin{align}\label{Levydf}
\lim\limits_{n\to \infty}\int_{|x|\le R_{n}}|u_{n}|^{2}dx=\int_{\R^{3}}|u|^{2}dx\in(0,1), \ \  \  \lim\limits_{n\to \infty}\int_{R_{n}\le |x|\le 6R_{n} }|u_{n}|^{2}dx=0.
\end{align}
Take a cut-off function $\chi\in C_{0}^{\infty}(\R^{3}, [0,1])$ such that  $\chi(x)=1$ for $|x|\le 1$ and $\chi(x)=0$ for $|x|\ge 2$. Define $\chi_{R_{n}}(x):=\chi(\frac{x}{R_{n}})$, $\eta_{R_{n}}(x):=\sqrt{1-\chi_{R_{n}}^{2}}$, and
\begin{equation}\label{eta}
{u}_{1n}:=\chi_{R_{n}}u_{n}, \ \, {u}_{2n}:=\eta_{R_{n}}u_{n}.
\end{equation}
It then follows  from \eqref{Levydf}  and Br\'ezis-Lieb Lemma (cf. \cite{minimax}) that
\begin{align*}\label{loccy}
u_{1n}\to u\  \  \text{and}\ \ \eta_{R_{n}}\chi_{3R_{n}} u_{n}\to 0 \  \ \text{strongly in} \   L^{2}(\R^{3},\C)\ \, \text{as}\ \, n\to \infty.
\end{align*}
Together with the interpolation inequality and the weak lower semicontinuity of norm, 
we further obtain that
\begin{align}
\liminf_{n\to \infty}\big\langle u_{1n}, \sqrt{-\Delta+m^{2}}u_{1n}\big\rangle\ge \big\langle u, \sqrt{-\Delta+m^{2}}u\big\rangle,
\end{align}
and
\begin{align}\label{a.7}
u_{1n}\to u,\  \ \ \eta_{R_{n}}\chi_{3R_{n}}u_{n}\to 0 \  \ \text{strongly in} \   L^{r}(\R^{3},\C)\  \text{as}\  n\to \infty, \ \  \forall\ r\in [2, 3).
\end{align}
Consequently, applying the IMS-type formula (cf. \cite[Lemma A.1]{Duke}), one can calculate from \eqref{eta}--\eqref{a.7} that
\begin{equation}\label{p1}
\begin{split}\lim\limits_{n\to \infty}\big\langle u_{n}, \sqrt{-\Delta+m^{2}}u_{n}\big\rangle
\ge \big\langle u, \sqrt{-\Delta+m^{2}}u\big\rangle+\lim\limits_{n\to \infty}\big\langle u_{2n},\sqrt{-\Delta+m^{2}}u_{2n}\big\rangle,
\end{split}\end{equation}
and
\begin{equation}\label{p3}
\begin{split}\lim\limits_{n\to \infty}\int_{\R^{3}}|u_{n}|^{\frac{8}{3}}dx
&=\lim\limits_{n\to \infty}\int_{\R^{3}}\Big(\chi^{2}_{R_{n}}|u_{n}|^{2}+\eta_{R_{n}}^{2}\chi^{2}_{3R_{n}}|u_{n}|^{2}+\eta_{3R_{n}}^{2}|u_{n}|^{2}
\Big)^{\frac{4}{3}}dx\\
&=\lim\limits_{n\to \infty}\int_{\R^{3}}\Big(\chi^{2}_{R_{n}}|u_{n}|^{2}+\eta_{3R_{n}}^{2}|u_{n}|^{2}
\Big)^{\frac{4}{3}}dx\\
&=\lim\limits_{n\to \infty}\int_{\R^{3}}(|u_{1n}|^{\frac{8}{3}}+|u_{2n}|^{\frac{8}{3}})dx\\
&=\int_{\R^{3}}|u|^{\frac{8}{3}}dx+\lim\limits_{n\to \infty}\int_{\R^{3}}|u_{2n}|^{\frac{8}{3}}dx.
\end{split}\end{equation}
Applying Lemma \ref{lem32} (2),
we then  deduce from \eqref{p1} and \eqref{p3} that
\begin{equation*}
\begin{split}E_{a}(\lam)+E_{a}(1-\lam)&\ge E_{a}(1)=\lim\limits_{n\to \infty}\mathcal{E}_{a}(|u_{n}\rangle \langle u_{n}|)\nonumber\\
&\ge \mathcal{E}_{a}(|u\rangle \langle u|)+ \lim\limits_{n\to \infty}\mathcal{E}_{a}(|u_{2n}\rangle \langle u_{2n}|)\nonumber\\
&\ge E_{a}(\lam)+\lim\limits_{n\to \infty}E_{a}(\|u_{2n}\|_2^2)\nonumber\\
&=E_{a}(\lam)+E_{a}(1-\lam),
\end{split}\end{equation*}
where the last identity follows from  the fact that $\lim\limits_{n\to \infty}\|u_{2n}\|_2^2=1-\lam\in(0,1)$.
This thus implies that $|u\rangle \langle u|$ is a minimizer of $E_{a}(\lam)$,  and
\begin{align}\label{id}
E_{a}(\lam)+E_{a}(1-\lam)=E_{a}(1).
\end{align}

Let $|v_{n}\rangle \langle v_{n}|$ be a minimizing sequence of $E_{a}(1-\lam)$, where $\lam=\|u\|_2^{2}\in (0,1)$. Since $E_{a}(1-\lam)<0$, we get that
\begin{align*}
a\lim\limits_{n\to \infty}\int_{\R^{3}}|v_{n}|^{\frac{8}{3}}dx=\lim\limits_{n\to \infty}\big\langle(v_n, \sqrt{-\Delta+m^{2}}-m)v_{n}\big\rangle-E_{a}(1-\lam)>0,
\end{align*}
and thus
\begin{align}\label{b343}
E_{a}(1)&\le \lim\limits_{n\to \infty} \mathcal{E}_{a}\big((1-\lam)^{-1}|v_{n}\rangle \langle v_{n}|\big) \nonumber\\
&=(1-\lam)^{-1}\Big[ E_{a}(1-\lam)-a\big((1-\lam)^{-\frac{1}{3}}-1\big)\lim\limits_{n\to \infty}\int_{\R^{3}}|v_{n}|^{\frac{8}{3}}dx
\Big]\\
&<(1-\lam)^{-1}E_{a}(1-\lam).\nonumber
\end{align}
Moreover, note that
\begin{align}\label{b342}
E_{a}(1)&\le \mathcal{E}_{a}(\lam^{-1}|u\rangle \langle u|) \nonumber\\
&=\lam^{-1}\Big[ E_{a}(\lam)-a(\lam^{-\frac{1}{3}}-1)\int_{\R^{3}}|u|^{\frac{8}{3}}dx
\Big]<\lam^{-1}E_{a}(\lam).
\end{align}
We then obtain  from (\ref{b343}) and  (\ref{b342}) that
\begin{align*}
E_{a}(1)<E_{a}(\lam)+E_{a}(1-\lam),
\end{align*}
which however contradicts with (\ref{id}), and  Lemma \ref{lem34} (1)  is therefore complete.

(2). Since the proof of Lemma \ref{lem34} (2) is similar to that of \cite[Theorem 27]{Geo} with minor modifications,  for simplicity we omit the detailed proof.  This completes the proof of Lemma \ref{lem34}.  \qed

\end{document}